\documentclass{amsart}

\usepackage{enumerate, amsmath, amsfonts, amssymb, amsthm, wasysym, graphics, graphicx, xcolor, url, hyperref, hypcap}
\hypersetup{colorlinks=true, citecolor=darkblue, linkcolor=darkblue}
\graphicspath{{figures/}}

\newtheorem{theorem}{Theorem}[section]
\newtheorem{corollary}[theorem]{Corollary}
\newtheorem{lemma}[theorem]{Lemma}
\newtheorem{proposition}[theorem]{Proposition}

\newtheorem{definition}[theorem]{Definition}

\newtheorem{question}[theorem]{Question}

\theoremstyle{definition}
\newtheorem{example}[theorem]{Example}
\newtheorem{remark}[theorem]{Remark}

\newcommand{\R}{\mathbb{R}}
\newcommand{\N}{\mathbb{N}}

\newcommand{\cN}{\mathcal{N}}

\newcommand{\cB}{\mathcal{B}}
\newcommand{\cP}{\mathcal{P}}
\newcommand{\cX}{\mathcal{X}}
\newcommand{\cY}{\mathcal{Y}}
\newcommand{\cZ}{\mathcal{Z}}
\newcommand{\fS}{\mathfrak{S}}
\newcommand{\set}[2]{\left\{ #1 \;\middle|\; #2 \right\}}
\newcommand{\GP}[2]{\mathbb{#1} \left(\left\{ #2 \right\}_{I\in[n]}\right)} 
\newcommand{\sGP}[2]{\mathbb{#1} (\{#2\})} 
\newcommand{\ssm}{\smallsetminus}
\newcommand{\dotprod}[2]{\left\langle #1 \;\middle|\; #2 \right\rangle}
\newcommand{\symdif}{\triangle}
\newcommand{\simplex}{\triangle}
\newcommand{\contact}{\#}
\newcommand{\one}{1\!\!1}
\newcommand{\VS}{\textsc{vs}}
\newcommand{\eqdef}{\mbox{\,\raisebox{0.2ex}{\scriptsize\ensuremath{\mathrm:}}\ensuremath{=}\,}} 
\newcommand{\eqfed}{\mbox{~\ensuremath{=}\raisebox{0.2ex}{\scriptsize\ensuremath{\mathrm:}} }} 
\DeclareMathOperator{\conv}{conv}
\DeclareMathOperator{\cone}{cone}
\DeclareMathOperator{\arr}{Arr}

\definecolor{darkblue}{rgb}{0,0,0.7} 
\newcommand{\darkblue}{\color{darkblue}}
\newcommand{\defn}[1]{\emph{\darkblue #1}} 
\newcommand{\fref}[1]{Figure~\ref{#1}} 
\newcommand{\ie}{\textit{i.e.}~} 
\newcommand{\eg}{\textit{e.g.}~} 
\renewcommand{\paragraph}[1]{\medskip\noindent\textsc{#1}.}

\begin{document}

\title{The brick polytope of a sorting network}
\thanks{
Research supported by grants MTM2008-04699-C03-02 and CSD2006-00032 (i-MATH) of the Spanish MICINN. An extended abstract of this paper appeared in the proceedings of the \emph{23rd International Conference on Formal Power Series and Algebraic Combinatorics} (FPSAC'11). The contents of Section \ref{sec:multi} also appeared in the PhD dissertation of the first author \cite{Pilaud}.
}

\author{Vincent Pilaud}
\address{Laboratoire d'Informatique de l'\'Ecole Polytechnique, Palaiseau, France}
\email{pilaud@lix.polytechnique.fr}

\author{Francisco Santos}
\address{Departamento de Matem\'aticas Estad\'istica y Computaci\'on, Universidad de Cantabria, Santander, Spain}
\email{francisco.santos@unican.es}

\date{}

\vspace*{-1.2cm}
\maketitle

\vspace*{-.9cm}
\begin{abstract}
The associahedron is a polytope whose graph is the graph of flips on triangulations of a convex polygon. Pseudotriangulations and multitriangulations generalize triangulations in two different ways, which have been unified by Pilaud \& Pocchiola in their study of flip graphs on pseudoline arrangements with contacts supported by a given sorting network.

In this paper, we construct the brick polytope of a sorting network, obtained as the convex hull of the brick vectors associated to each pseudoline arrangement supported by the network. We combinatorially characterize the vertices of this polytope, describe its faces, and decompose it as a Minkowski sum of matroid polytopes.

Our brick polytopes include Hohlweg \& Lange's many realizations of the associahedron, which arise as brick polytopes for certain well-chosen sorting networks. We furthermore discuss the brick polytopes of sorting networks supporting pseudoline arrangements which correspond to multitriangulations of convex polygons: our polytopes only realize subgraphs of the flip graphs on multitriangulations and they cannot appear as projections of a hypothetical multiassociahedron.

\medskip
\noindent
{\sc keywords.}
associahedron $\cdot$ sorting networks $\cdot$ pseudoline arrangements with contacts $\cdot$ multitriangulations $\cdot$ zonotopes
\end{abstract}

\tableofcontents
\vspace{-1cm}


\newpage
\section{Introduction}

This paper focusses on polytopes realizing flip graphs on certain geometric and combinatorial structures. Various examples of such polytopes are illustrated in \fref{fig:polytopalRealizations} and described along this introduction. The motivating example is the \defn{associahedron} whose vertices correspond to triangulations of a convex polygon~$\cP$ and whose edges correspond to flips between them. The boundary complex of its polar is (isomorphic to) the simplicial complex of crossing-free sets of internal diagonals of~$\cP$. The associahedron appears under various motivations ranging from geometric combinatorics to algebra, and several different constructions have been proposed (see~\cite{Lee,Loday,HohlwegLange,CeballosSantosZiegler}). We have represented two different realizations of the $3$-dimensional associahedron in \fref{fig:polytopalRealizations} (top). In fact, the associahedron is a specific case of a more general polytope: the \defn{secondary polytope}~\cite{gkz,bfs} of a $d$-dimensional set~$\cP$ of $n$ points is a $(n-d-1)$-dimensional polytope whose vertices correspond to regular triangulations of~$\cP$ and whose edges correspond to regular flips between them. Its boundary complex is (isomorphic to) the refinement poset of regular polyhedral subdivisions of~$\cP$. See \fref{fig:polytopalRealizations}~(middle left). We refer to~\cite{lrs} for a reference on triangulations of point sets and of the structure of their flip graphs.

\begin{figure}[b]
  \capstart
  \centerline{\includegraphics[width=.95\textwidth]{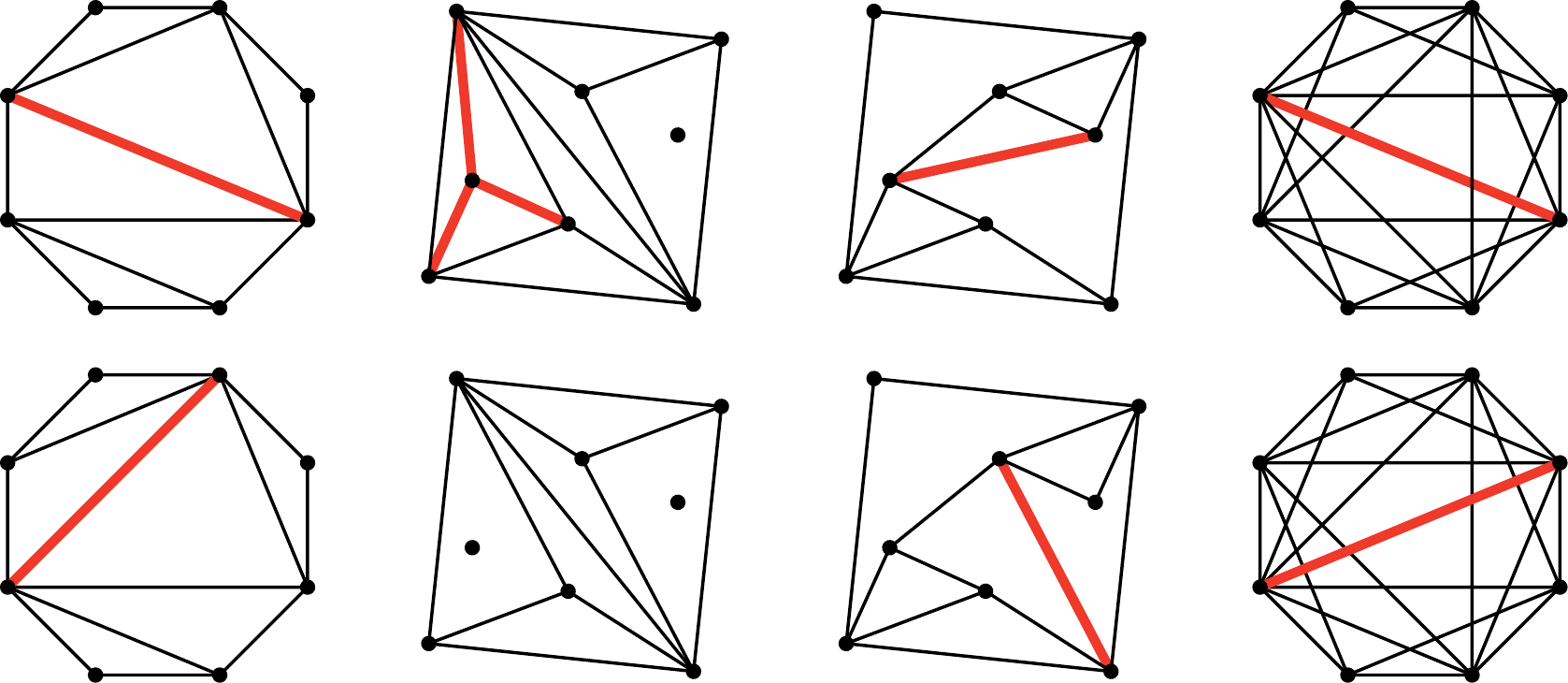}}
  \caption{Flips in four geometric structures: a triangulation of a convex polygon, a triangulation of a general point set, a pseudotriangulation and a $2$-triangulation of a convex polygon.}
  \label{fig:geometricFlips}
\end{figure}

\begin{figure}[p]
  \capstart
  \centerline{\includegraphics[width=1.08\textwidth]{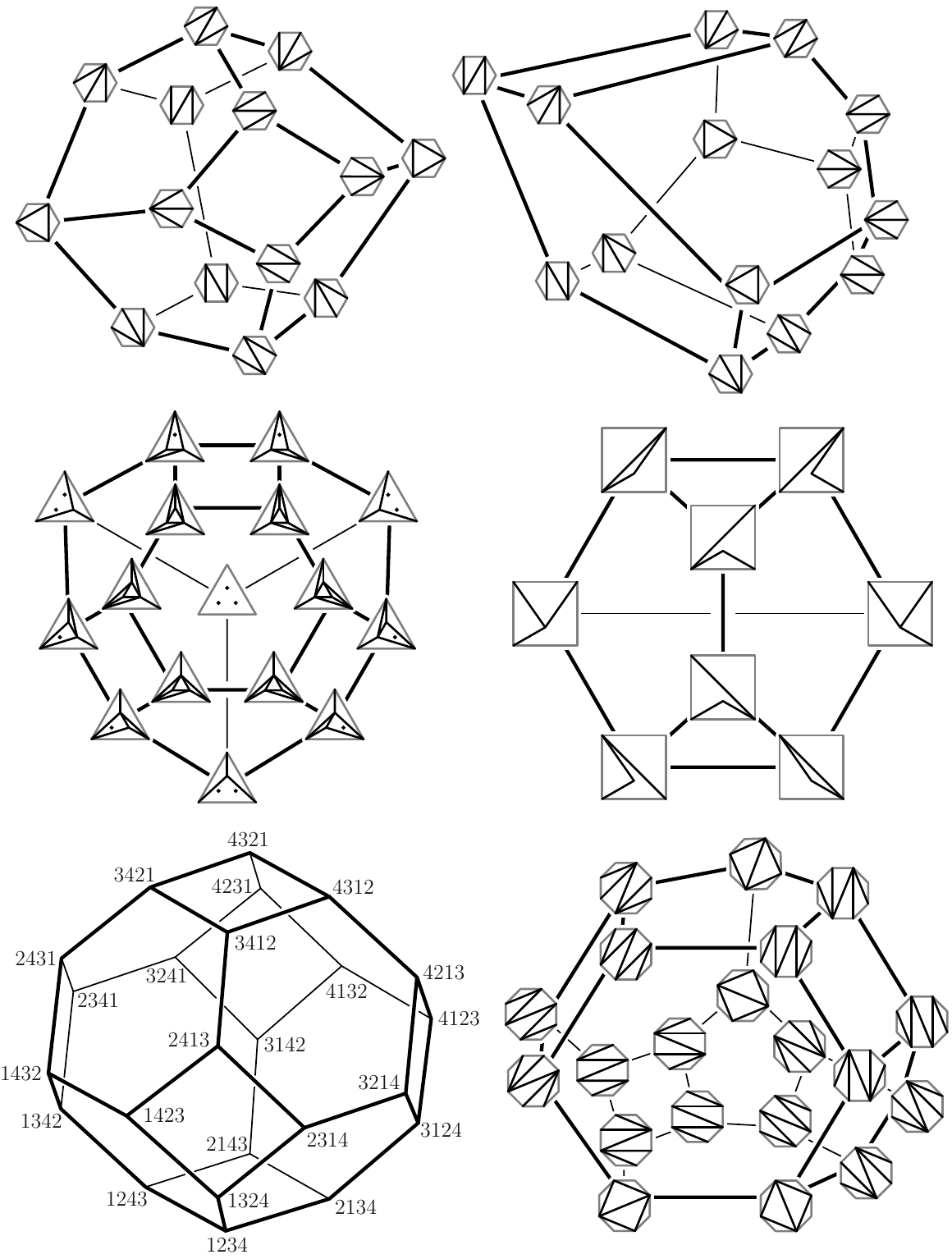}}
  \caption{Polytopal realizations of various flip graphs: two constructions of the $3$-dimensional associahedron ---~the secondary polytope of the regular hexagon (top left) and Loday's construction (top right); the secondary polytope of a set of $6$ points (middle left); a $3$-dimensional pseudotriangulations polytope (middle right); the $3$-dimensional permutahedron (bottom left); the $3$-dimensional cyclohedron (bottom right).}
  \label{fig:polytopalRealizations}
\end{figure}

Our work was motivated by two different generalizations of planar triangulations, whose combinatorial structures extend that of the associahedron~---~see \fref{fig:geometricFlips}:
\begin{enumerate}[(i)]
\item Let~$\cP$ be a point set in general position in the Euclidean plane, with $i$ interior points and $b$ boundary points. A set of edges with vertices in~$\cP$ is \defn{pointed} if the edges incident to any point of~$\cP$ span a pointed cone. A \defn{pseudotriangle} on~$\cP$ is a simple polygon with vertices in~$\cP$, which has precisely three convex corners, joined by three concave chains. A (pointed) \defn{pseudotriangulation} of~$\cP$ is a decomposition of its convex hull into~$i+b-2$ pseudotriangles on~$\cP$~\cite{PocchiolaVegter,RoteSantosStreinu-survey}. Equivalently, it is a maximal pointed and crossing-free set of edges with vertices in~$\cP$. The \defn{pseudotriangulations polytope}~\cite{RoteSantosStreinu-polytope} of the point set~$\cP$ is a simple $(2i+b-3)$-dimensional polytope whose vertices correspond to pseudotriangulations of~$\cP$ and whose edges correspond to flips between them. The boundary complex of its polar is (isomorphic to) the simplicial complex of pointed crossing-free sets of internal edges on~$\cP$. See \fref{fig:polytopalRealizations} (middle right).
\item A \defn{$k$-triangulation} of a convex $n$-gon~$\cP$ is a maximal set of diagonals with no $(k+1)$-crossing (no $k+1$ diagonals are mutually crossing)~\cite{CapoyleasPach, PilaudSantos, Pilaud}. We can forget the diagonals of the $n$-gon with less than $k$ vertices of~$\cP$ on one side: they cannot appear in a $(k+1)$-crossing and thus they belong to all $k$-triangulations of~$\cP$. The other edges are called \defn{$k$-relevant}. The simplicial complex~$\Delta_n^k$ of $(k+1)$-crossing-free sets of $k$-relevant diagonals of $\cP$ is a topological sphere~\cite{Jonsson, Stump} whose facets are $k$-triangulations of $\cP$ and whose ridges are flips between them. It remains open whether or not this simplicial complex is the boundary complex of a polytope. 
\end{enumerate}

In~\cite{PilaudPocchiola}, Pilaud \& Pocchiola developed a general framework which generalizes both pseudotriangulations and multitriangulations. They study the graph of flips on \defn{pseudoline arrangements with contacts} supported by a given sorting network. The present paper is based on this framework. Definitions and basic properties are recalled in Section~\ref{sec:definitions}.

In this paper, we define and study the \defn{brick polytope} of a sorting network~$\cN$, obtained as the convex hull of vectors associated to each pseudoline arrangement supported by~$\cN$. Our main result is the characterization of the pseudoline arrangements which give rise to the vertices of the brick polytope, from which we derive a combinatorial description of the faces of the brick polytope. We furthermore provide a natural decomposition of the brick polytope into a Minkowski sum of matroid polytopes. These structural results are presented in Section~\ref{sec:structure}. We illustrate the results of this section with particular sorting networks whose brick polytopes are graphical zonotopes. Among them, the \defn{permutahedron} is a well-known simple $(n-1)$-dimensional polytope whose vertices correspond to permutations of~$[n]$ and whose edges correspond to pairs of permutations which differ by an adjacent transposition. Its boundary complex is (isomorphic to) the refinement poset of ordered partitions of~$[n]$. See \fref{fig:polytopalRealizations} (bottom left).

We obtain our most relevant examples of brick polytopes in Section~\ref{sec:associahedra}. We observe that for certain well-chosen sorting networks, our brick polytopes coincide (up to translation) with Hohlweg \& Lange's realizations of the associahedron~\cite{HohlwegLange}. We therefore provide a complementary point of view on their polytopes and we complete their combinatorial description. We obtain in particular a natural Minkowski sum decomposition of these polytopes into matroid polytopes.

Finally, Section~\ref{sec:multi} is devoted to our initial motivation for the construction of the brick polytope. We wanted to find a polytopal realization of the simplicial complex~$\Delta_n^k$ of $(k+1)$-crossing-free sets of $k$-relevant diagonals of the $n$-gon. Using Pilaud \& Pocchiola's correspondence between multitriangulations and pseudoline arrangement covering certain sorting networks~\cite{PilaudPocchiola}, we construct a point configuration in $\R^{n-2k}$ with one point associated to each $k$-triangulation of the $n$-gon. We had good reasons to believe that this point set could be a projection of the polar of a realization of~$\Delta_n^k$: the graph of the corresponding brick polytope (the convex hull of this point configuration) is a subgraph of flips, and all sets of $k$-triangulations whose corresponding points belong to a given face of this brick polytope are faces of~$\Delta_n^k$. However, we prove  that our point configuration cannot be a projection of the polar of a realization of~$\Delta_n^k$.

After the completion of a preliminary version of this paper, Stump pointed out to us his paper~\cite{Stump} which connects the multitriangulations to the type~$A$ subword complexes of Knutson \& Miller~\cite{KnutsonMiller}. The latter can be visually interpreted as sorting networks (see Section~\ref{subsec:subwordComplex}). This opened the perspective of the generalization of brick polytopes to subword complexes on Coxeter groups. This generalization was achieved by Pilaud \& Stump in~\cite{PilaudStump}. This construction yields in particular the generalized associahedra of Hohlweg, Lange \& Thomas~\cite{HohlwegLangeThomas} for certain particular subword complexes described by Ceballos, Labb\'e \& Stump~\cite{CeballosLabbeStump}. In the present paper, we focus on the classical situation of type~$A$, which already reflects the essence of the construction. The only polytope of different type which appears here is the \defn{cyclohedron} via its standard embedding in the associahedron. The vertices of the cyclohedron correspond to centrally symmetric triangulations of a centrally symmetric convex $(2n)$-gon and its edges correspond to centrally symmetric flips between them (\ie either a flip of a centrally symmetric diagonal or a simultaneous flip of a pair of symmetric diagonals). The boundary complex of its polar is (isomorphic to) the refinement poset of centrally symmetric polygonal subdivisions of the $(2n)$-gon. See \fref{fig:polytopalRealizations} (bottom right).

We moreover refer to~\cite{PilaudStump} for further properties of the brick polytope which appeared when generalizing it to Coxeter groups of finite types. They relate in particular the graph of the brick polytope to a quotient of the weak order and the normal fan of the brick polytope to the Coxeter fan.


\section{The brick polytope of a sorting network}\label{sec:definitions}

\subsection{Pseudoline arrangements on sorting networks}\label{subsec:networks}

Consider a set of $n$ hori\-zontal lines (called \defn{levels}, and labeled from bottom to top), and place $m$ vertical segments (called \defn{commutators}, and labeled from left to right) joining two consecutive horizontal lines, such that no two commutators have a common endpoint~---~see \eg Figure~\ref{fig:network}. Throughout this paper, we fix such a configuration $\cN$ that we call a \defn{network}. The \defn{bricks} of~$\cN$ are its $m-n+1$~bounded cells. We say that a network is \defn{alternating} when the commutators adjacent to each intermediate level are alternatively located above and below it.

\begin{figure}[ht]
  \capstart
  \centerline{\includegraphics[width=\textwidth]{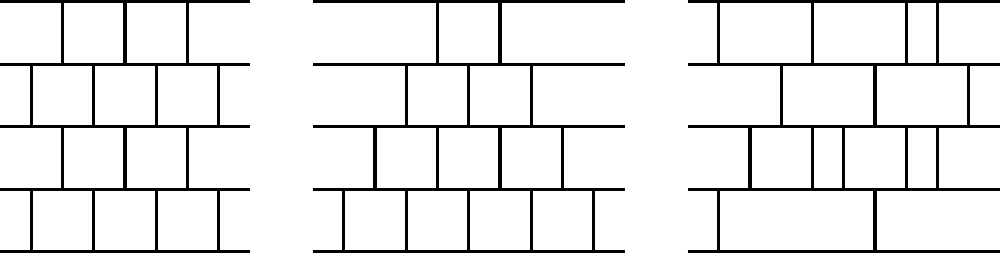}}
  \caption{Three networks with $5$ levels, $14$ commutators and $10$ bricks. The first two are alternating.}
  \label{fig:network}
\end{figure}

A \defn{pseudoline} is an abscissa monotone path on the network $\cN$. A \defn{contact} between two pseudolines is a commutator whose endpoints are contained one in each pseudoline, and a \defn{crossing} between two pseudolines is a commutator traversed by both pseudolines. A \defn{pseudoline arrangement} (with contacts) is a set of $n$ pseudolines supported by $\cN$ such that any two of them have precisely one crossing, some (perhaps zero) contacts, and no other intersection~---~see Figure~\ref{fig:flip}. Observe that in a pseudoline arrangement, the pseudoline which starts at level $\ell$ necessarily ends at level $n+1-\ell$ and goes up at $n-\ell$ crossings and down at $\ell-1$ crossings. Note also that a pseudoline arrangement supported by $\cN$ is completely determined by its ${n \choose 2}$ crossings, or equivalently by its $m-{n \choose 2}$ contacts. Let~$\arr(\cN)$ denote the set of pseudoline arrangements supported by~$\cN$. We say that a network is \defn{sorting} when it supports at least one pseudoline arrangement.

\begin{figure}
  \capstart
  \centerline{\includegraphics[width=\textwidth]{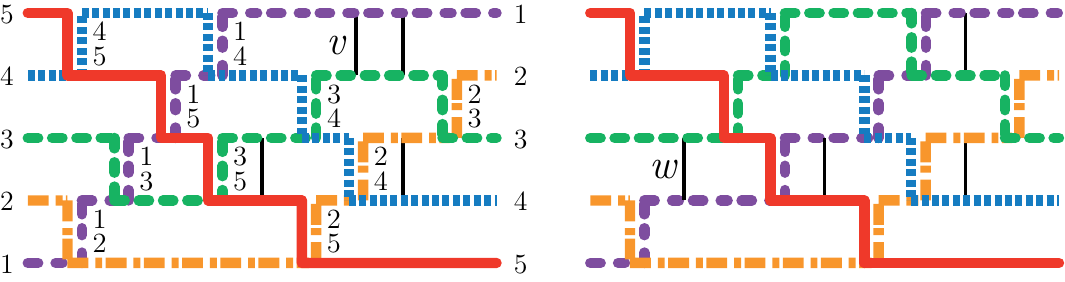}}
  \caption{Two pseudoline arrangements, both supported by the rightmost network~$\cN$ of \fref{fig:network}, and related by a flip. The left one is the greedy pseudoline arrangement~$\Gamma(\cN)$, whose flips are all decreasing. It is obtained by sorting the permutation~$(5,4,3,2,1)$ according to the network~$\cN$.}
  \label{fig:flip}
\end{figure}

\subsection{The graph of flips}\label{subsec:flips}

There is a natural \defn{flip} operation which transforms a pseudoline arrangement supported by $\cN$ into another one by exchanging the position of a contact. More precisely, if $V$ is the set of contacts of a pseudoline arrangement $\Lambda$ supported by~$\cN$, and if $v\in V$ is a contact between two pseudolines of $\Lambda$ which cross at $w$, then $(V\ssm\{v\}) \cup \{w\}$ is the set of contacts of another pseudoline arrangement supported by $\cN$~---~see Figure~\ref{fig:flip}. The \defn{graph of flips}~$G(\cN)$ is the graph whose nodes are the pseudoline arrangements supported by~$\cN$ and whose edges are the flips between them. This graph was studied in~\cite{PilaudPocchiola}, whose first statement is the following result:

\begin{theorem}[{\cite{PilaudPocchiola}}]\label{theo:connect}
The graph of flips $G(\cN)$ of a sorting network~$\cN$ with $n$ levels and $m$ commutators is $\left(m-{n \choose 2}\right)$-regular and connected.
\end{theorem}

Regularity of the graph of flips is obvious since every contact induces a flip. For the connectivity, define a flip to be \defn{decreasing} if the added contact lies on the left of the removed contact. The oriented graph of decreasing flips is clearly acyclic and is proved to have a unique source in~\cite{PilaudPocchiola} (and thus, to be connected). This source is called the \defn{greedy pseudoline arrangement} supported by~$\cN$ and is denoted by~$\Gamma(\cN)$. It is characterized by the property that any of its contacts is located to the right of its corresponding crossing. It can be computed by sorting the permutation~$(n,n-1,\dots,2,1)$ according to the sorting network~$\cN$~---~see \fref{fig:flip}~(left). We will use this particular pseudoline arrangement later. We refer to~\cite{PilaudPocchiola} for further details.

For any given subset~$\gamma$ of the commutators of~$\cN$, we denote by $\arr(\cN|\gamma)$ the set of pseudoline arrangements supported by~$\cN$ and whose set of contacts contains~$\gamma$. The arrangements of $\arr(\cN|\gamma)$ are in obvious correspondence with that of $\arr(\cN\ssm\gamma)$, where~$\cN\ssm\gamma$ denotes the network obtained by erasing the commutators of~$\gamma$ in~$\cN$. In particular, the subgraph of~$G(\cN)$ induced by~$\arr(\cN|\gamma)$ is isomorphic to $G(\cN\ssm\gamma)$, and thus Theorem~\ref{theo:connect} ensures that this subgraph is connected for every~$\gamma$.

More generally, let~$\Delta(\cN)$ denote the simplicial complex of all sets of commutators of~$\cN$ contained in the set of contacts of a pseudoline arrangement supported by~$\cN$. In other words, a set~$\gamma$ of commutators of~$\cN$ is a face of $\Delta(\cN)$ if and only if the network $\cN\ssm\gamma$ is still sorting. This complex is pure of dimension~$m-{n \choose 2}-1$, its maximal cells correspond to pseudoline arrangements supported by~$\cN$ and its ridge graph is the graph of flips~$G(\cN)$. The previous connectedness properties ensure that~$\Delta(\cN)$ is an abstract polytope~\cite{Schulte}, and it is even a combinatorial sphere (see Corollary~\ref{coro:shellable} and the discussion in Section~\ref{subsec:multiassociahedron}). These properties motivate the following question:

\begin{question}\label{qu:polytope}
Is~$\Delta(\cN)$ the boundary complex of a $\left(m-{n \choose 2}\right)$-dimensional simplicial polytope?
\end{question}

In this article, we construct a polytope whose graph is a subgraph of~$G(\cN)$, and which combinatorially looks like ``a projection of'' the dual complex of $\Delta(\cN)$. More precisely, we associate a vector $\omega(\Lambda) \in \R^n$ to each~$\Lambda\in\arr(\cN)$, and we consider the convex hull~${\Omega(\cN) \eqdef \conv\set{\omega(\Lambda)}{\Lambda\in\arr(\cN)} \subset\R^n}$ of all these vectors. The resulting polytope has the property that for every face~$F$ of~$\Omega(\cN)$ there is a set~$\gamma$ of commutators of~$\cN$ such that $\arr(\cN|\gamma) = \set{\Lambda\in\arr(\cN)}{\omega(\Lambda)\in F}$. In particular, when the dimension of~$\Omega(\cN)$ is $m-{n \choose 2}$, our construction answers Question~\ref{qu:polytope} in the affirmative. The relationship between our construction and Question~\ref{qu:polytope} is discussed in more details in Section~\ref{subsec:multiassociahedron}.

\subsection{Subword complexes on finite Coxeter groups}
\label{subsec:subwordComplex}

Before presenting our construction of the brick polytope of a sorting network, we make a little detour to connect the abovementioned simplicial complexes $\Delta(\cN)$ with the subword complexes of Knutson \& Miller~\cite{KnutsonMiller}. 

Let $(W,S)$ be a finite Coxeter system, that is, $W$ is a finite reflection group and $S$ is a set of simple reflections minimally generating $W$. See for example~\cite{Humphreys} for background. Let $Q$ be a word on the alphabet $S$ and let $\rho$ be an element of~$W$. The \defn{subword complex} $\Delta(Q,\rho)$ is the pure simplicial complex of subwords of~$Q$ whose complements contain a reduced expression of $\rho$~\cite{KnutsonMiller}. The vertices of this simplicial complex are labeled by the positions in the word $Q$ (note that two positions are different even if the letters of $Q$ at these positions coincide), and its facets are the complements of the reduced expressions of $\rho$ in the word $Q$.

There is a straightforward combinatorial isomorphism between:
\begin{enumerate}[(i)]
\item the simplicial complex $\Delta(\cN)$ discussed in the previous section, where $\cN$ is a sorting network on $n$ levels and $m$ commutators; and
\item the subword complex $\Delta(Q,w_\circ)$, where the underlying Coxeter group $W$ is the symmetric group $\fS_n$ on $n$ elements, the simple system $S$ is the set of adjacent transpositions $\tau_i \eqdef (i,i+1)$, the word $Q = \tau_{i_1}\tau_{i_2}\dots\tau_{i_m}$ is formed according to the positions of the commutators of~$\cN$ --- the $j$\textsuperscript{th} leftmost commutator of $\cN$ lies between the $i_j$\textsuperscript{th} and $(i_j+1)$\textsuperscript{th} levels of $\cN$ ---, and the permutation ${w_\circ = [n,n-1,\dots,2,1]}$ is the longest element of $\fS_n$ --- its reduced expressions on $S$ all have the maximal length ${n \choose 2}$.
\end{enumerate}

Since Pilaud \& Pocchiola~\cite{PilaudPocchiola} were not aware of the definition of the subword complex, they studied the simplicial complexes $\Delta(\cN)$ independently and rediscovered some relevant properties which hold for any subword complex. The connection between multitriangulations and subword complexes was first done by Stump~\cite{Stump}, providing the powerful toolbox of Coxeter combinatorics to the playground.

In particular, Knutson \& Miller prove in~\cite{KnutsonMiller} that the subword complex $\Delta(Q,\rho)$ is either a combinatorial sphere or a combinatorial ball, depending on whether the Demazure product of $Q$ equals $\rho$ or not. The interested reader can refer to their article for details on this property and for other known properties on subword complexes. In the conclusion of their article, Question~6.4 asks in particular whether any spherical subword complex is the boundary complex of a convex simplicial polytope, which is a generalized version of Question~\ref{qu:polytope} stated above.

Throughout our article, we only consider subword complexes on the Coxeter system $(\fS_n, \set{\tau_i}{i\in[n-1]})$ and with $\rho=w_\circ$. However, we want to mention that Pilaud \& Stump~\cite{PilaudStump} extended the construction of this paper to any subword complex on any Coxeter system. This generalized construction yields in particular the generalized associahedra of Hohlweg, Lange \& Thomas~\cite{HohlwegLangeThomas} for certain particular subword complexes described by Ceballos, Labb\'e \& Stump~\cite{CeballosLabbeStump}.

\subsection{The brick polytope}

The subject of this paper is the following polytope:

\begin{definition}
Let $\cN$ be a sorting network with $n$ levels. The \defn{brick vector} of a pseudoline arrangement $\Lambda$ supported by $\cN$ is the vector $\omega(\Lambda) \in \R^n$ whose $i$\textsuperscript{th} coordinate is the number of bricks of~$\cN$ located below the $i$\textsuperscript{th} pseudoline of $\Lambda$ (the one which starts at level $i$ and finishes at level $n+1-i$). The \defn{brick polytope} $\Omega(\cN) \subset \R^n$ of the sorting network $\cN$ is the convex hull of the brick vectors of all pseudoline arrangements supported by $\cN$:
$$\Omega(\cN) \eqdef \conv\set{\omega(\Lambda)}{\Lambda\in\arr(\cN)} \subset\R^n.$$
\end{definition}

This article aims to describe the combinatorial properties of this polytope in terms of the properties of the supporting network. In Section~\ref{sec:structure}, we provide a characterization of the pseudoline arrangements supported by~$\cN$ whose brick vectors are vertices of the brick polytope~$\Omega(\cN)$, from which we derive a combinatorial description of the faces of the brick polytope. We also provide a natural decomposition of~$\Omega(\cN)$ into a Minkowski sum of simpler polytopes. In Section~\ref{sec:associahedra}, we recall the duality between the triangulations of a convex polygon and the pseudoline arrangements supported by certain networks~\cite{PilaudPocchiola}, whose brick polytopes coincide with Hohlweg \& Lange's realizations of the associahedron~\cite{HohlwegLange}. We finally discuss in Section~\ref{sec:multi} the properties of the brick polytopes of more general networks which support pseudoline arrangements corresponding to multitriangulations of convex polygons~\cite{PilaudSantos,PilaudPocchiola}.

We start by observing that the brick polytope is not full dimensional. Define the \defn{depth} of a brick of $\cN$ to be the number of levels located above it, and let $D(\cN)$ be the sum of the depths of all the bricks of~$\cN$. Since any pseudoline arrangement supported by $\cN$ covers each brick as many times as its depth, all brick vectors are contained in the following hyperplane:

\begin{lemma}\label{lem:depth}
The brick polytope~$\Omega(\cN)\subset\R^n$ is contained in the hyperplane of equation $\sum_{i=1}^n x_i = D(\cN)$.
\end{lemma}

The dimension of~$\Omega(\cN)$ is thus at most $n-1$, but could be smaller. We obtain the dimension of~$\Omega(\cN)$ in Corollary~\ref{coro:dim}.

We can also describe immediately the action of the vertical and horizontal reflections of the network on the brick polytope. 
The brick polytope of the network~$v(\cN)$ obtained by reflecting~$\cN$ through the vertical axis is the image of~$\Omega(\cN)$ under the affine transformation $(x_1,\dots,x_n) \mapsto (x_n,\dots,x_1)$. Similarly, the brick polytope of the network~$h(\cN)$ obtained by reflecting~$\cN$ through the horizontal axis is the image of~$\Omega(\cN)$ under the affine transformation $(x_1,\dots,x_n) \mapsto (m-n+1)\one - (x_n,\dots,x_1)$.

\subsection{Examples}\label{subsec:examples}

Before going on, we present some examples which will illustrate our results throughout the paper. Further motivating examples will be studied in Sections~\ref{sec:associahedra} and~\ref{sec:multi}.

\begin{example}[Reduced networks]\label{exm:reduced}
A sorting network~$\cN$ with~$n$ levels and~${m={n \choose 2}}$ commutators supports a unique pseudoline arrangement. Consequently, the graph of flips~$G(\cN)$, the simplicial complex~$\Delta(\cN)$ and the brick polytope~$\Omega(\cN)$ are all reduced to a single point. Such a network is said to be \defn{reduced}.
\end{example}

\begin{example}[$2$-level networks]
Consider the network $\cX_m$ formed by two levels related by $m$ commutators. We obtain a pseudoline arrangement by choosing any of these commutators as the unique crossing between two pseudolines supported by~$\cX_m$. Thus, the graph of flips~$G(\cX_m)$ is the complete graph on $m$ vertices, and the simplicial complex~$\Delta(\cX_m)$ is a $(m-1)$-dimensional simplex. The brick polytope $\Omega(\cX_m)$ is, however, a segment.
\begin{figure}[ht]
  \capstart
  \centerline{\includegraphics[width=.8\textwidth]{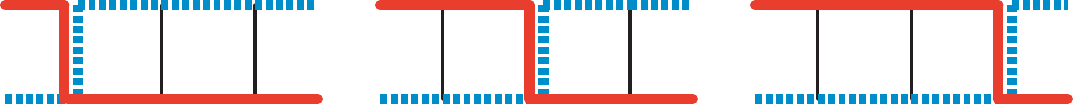}}
  \caption{The three pseudoline arrangements supported by the network~$\cX_3$ with two levels and three commutators.}
  \label{fig:2levels}
\end{figure}

\noindent
The brick vector of the pseudoline arrangement whose crossing is the $i$\textsuperscript{th} commutator of~$\cX_m$ is the vector $(m-i,i-1)$. Thus, the brick polytope~$\Omega(\cX_m)$ is the segment from $(m-1,0)$ to $(0,m-1)$. Its endpoints are the brick vectors of the pseudoline arrangements whose crossings are respectively the first and the last commutator of~$\cX_m$. The former is the source (\ie the greedy pseudoline arrangement) and the latter is the sink in the oriented graph of decreasing flips. The polytope $\Omega(\cX_m)$ is contained in the hyperplane of equation $x+y=m-1$.
\end{example}

\begin{example}[$3$-level alternating networks]\label{exm:3levels}
Let~$m \ge 3$ be an odd integer. Consider the network~$\cY_m$ formed by 3 levels related by $m$ alternating commutators. Any choice of~$3$ alternating crossings provides a pseudoline arrangement supported by~$\cY_m$. Consequently, $\cY_m$ supports precisely $\frac{1}{24}(m-1)m(m+1)$ pseudoline arrangements. The brick polytope~$\Omega(\cY_m)$ is a single point when $m=3$, a pentagon when $m=5$, and a hexagon for any $m \ge 7$. We have represented in \fref{fig:3levels} the projection of~$\Omega(\cY_{11})$ on the first and third coordinates plane, in such a way that the transformation $(x_1,x_3) \mapsto (x_3,x_1)$ appears as a vertical reflection.

\begin{figure}
  \capstart
  \centerline{\includegraphics[width=.85\textwidth]{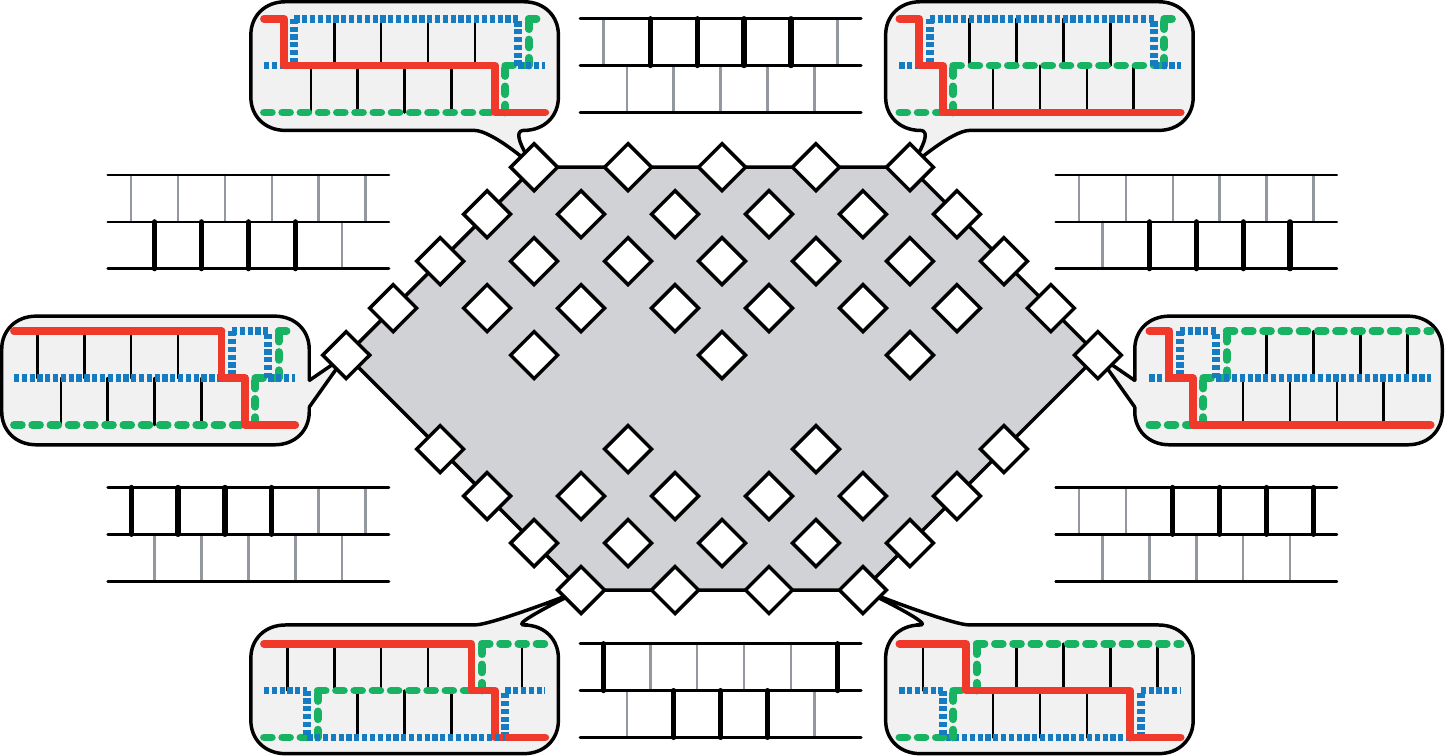}}
  \caption{The brick polytope~$\Omega(\cY_{11})$ of a $3$-level alternating network, projected on the first and third coordinates plane. Next to each vertex is drawn the corresponding pseudoline arrangement, and next to each edge is drawn the corresponding set of contacts.}
  \label{fig:3levels}
\end{figure}
\end{example}

\begin{example}[Duplicated networks]\label{exm:duplicated1}
Consider a reduced network~$\cN$ with $n$ levels and ${n \choose 2}$ commutators. For any distinct~$i,j\in[n]$, we labeled by $\{i,j\}$ the commutator of~$\cN$ where the $i$\textsuperscript{th} and $j$\textsuperscript{th} pseudolines of the unique pseudoline arrangement supported by~$\cN$ cross. 
Let $\Gamma$ be a connected graph on~$[n]$. We define $\cZ(\Gamma)$ to be the network with $n$ levels and $m={n \choose 2} + |\Gamma|$ commutators obtained from~$\cN$ by duplicating the commutators labeled by the edges of~$\Gamma$~---~see \fref{fig:duplicated}. We say that~$\cZ(\Gamma)$ is a \defn{duplicated network}. Observe that a pseudoline arrangement supported by~$\cZ(\Gamma)$ has a crossing for each commutator which has not been duplicated, and a crossing and a contact among each pair of duplicated commutators. Thus, $\cZ(\Gamma)$ supports precisely $2^{|\Gamma|}$ different pseudoline arrangements, the graph of flips~$G(\cZ(\Gamma))$ is the graph of the $|\Gamma|$-dimensional cube, and more generally, the simplicial complex $\Delta(\cZ(\Gamma))$ is the boundary complex of the $|\Gamma|$-dimensional cross-polytope. As an application of the results of Section~\ref{sec:structure}, we will see that the brick polytope of the duplicated network~$\cZ(\Gamma)$ is a graphical zonotope~---~see Examples~\ref{exm:duplicated2},~\ref{exm:duplicated3},~\ref{exm:duplicated4}, and~\ref{exm:duplicated5}. In particular, when $\Gamma$ is complete we obtain the permutahedron, while when~$\Gamma$ is a tree, we obtain a cube.

\begin{figure}[ht]
  \capstart
  \centerline{\includegraphics[scale=.48]{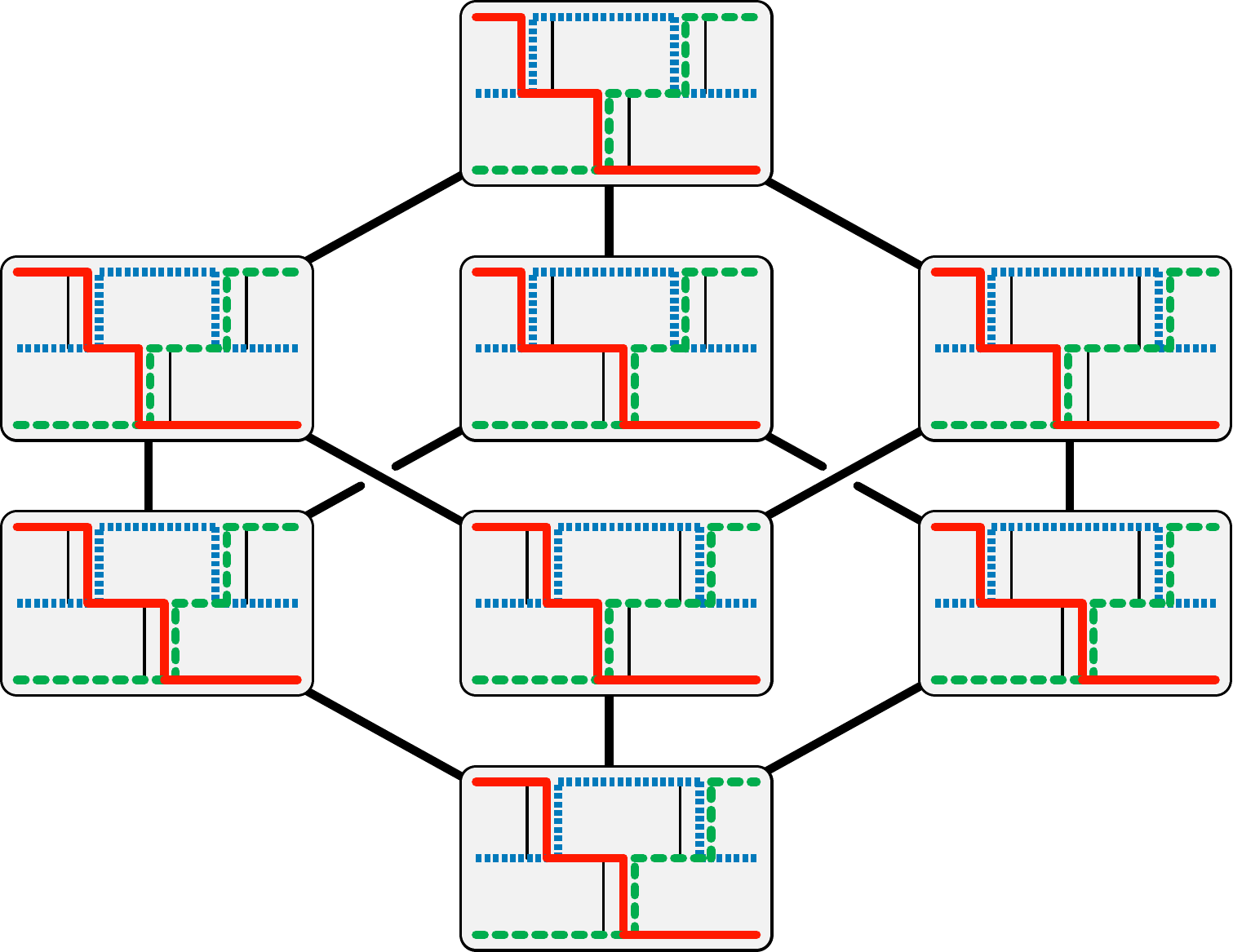}}
  \caption{The graph of flips of the duplicated network $\cZ(K_3)$.}
  \label{fig:duplicated}
\end{figure}
\end{example}


\section{Combinatorial description of the brick polytope}\label{sec:structure}

In this section, we characterize the vertices and describe the faces of the brick polytope $\Omega(\cN)$. For this purpose, we study the cone of the brick polytope~$\Omega(\cN)$ at the brick vector of a given pseudoline arrangement supported by~$\cN$. Our main tool is the incidence configuration of the contact graph of a pseudoline arrangement, which we define next.

\subsection{The contact graph of a pseudoline arrangement}

Let $\cN$ be a sorting network with~$n$ levels and~$m$ commutators, and let~$\Lambda$ be a pseudoline arrangement supported by~$\cN$.

\begin{definition}
The \defn{contact graph} of $\Lambda$ is the directed multigraph $\Lambda^\contact$ with a node for each pseudoline of $\Lambda$ and an arc for each contact of $\Lambda$ oriented from the pseudoline passing above the contact to the pseudoline passing below it.
\end{definition}

\begin{figure}[h]
  \capstart
  \centerline{\includegraphics[width=\textwidth]{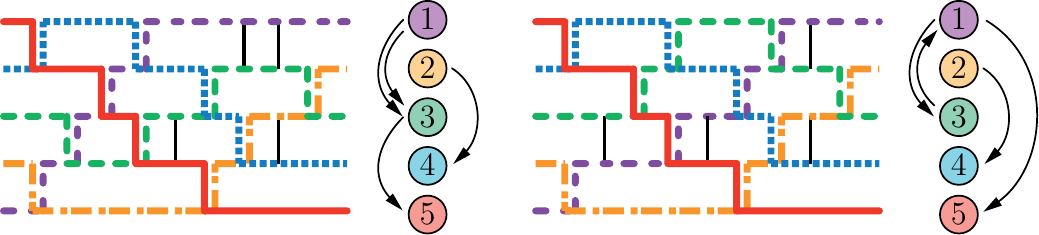}}
  \caption{The contact graphs of the pseudoline arrangements of \fref{fig:flip}. The connected components are preserved by the flip.}
  \label{fig:contact}
\end{figure}

The nodes of the contact graph come naturally labeled by~$[n]$: we label by~$\ell$ the node corresponding to the pseudoline of~$\Lambda$ which starts at level~$\ell$ and finishes at level~$n+1-\ell$. With this additional labeling, the contact graph provides enough information to characterize its pseudoline arrangement:

\begin{lemma}\label{lem:contactgraph}
Let~$\cN$ be a network and~$\Lambda^\contact$ be a graph on~$[n]$. If $\Lambda^\contact$ is the contact graph of a pseudoline arrangement~$\Lambda$ supported by~$\cN$, then~$\Lambda$ can be reconstructed from~$\Lambda^\contact$ and~$\cN$.
\end{lemma}

\begin{proof}
To obtain a pseudoline arrangement from its contact graph~$\Lambda^\contact$, we sort the permutation~$(n,n-1,\dots,2,1)$ on~$\cN$ according to~$\Lambda^\contact$. We sweep the network from left to right, and start to draw the $\ell$\textsuperscript{th} pseudoline at level~$\ell$. When we reach a commutator of~$\cN$ with pseudoline~$i$ above and pseudoline~$j$ below, 
\begin{itemize}
\item if there remains an arc $(i,j)$ in~$\Lambda^\contact$, we insert a contact in place of our commutator and delete an arc~$(i,j)$ from~$\Lambda^\contact$;
\item otherwise, we insert a crossing in place of our commutator: the indices $i$ and $j$ of the permutation get sorted at this crossing.
\end{itemize}
This procedures is correct since the contacts between two pseudolines $i<j$ are all directed from $i$ to $j$ before their crossing and from $j$ to $i$ after their crossing.
\end{proof}

\begin{example}\label{exm:contactgreedy}
We already mentioned a relevant example of the sorting procedure of the previous proof: the greedy pseudoline arrangement~$\Gamma(\cN)$ is obtained by sorting the permutation~$(n,n-1,\dots,2,1)$ on~$\cN$ and inserting the crossings as soon as possible. In other words any arc of the contact graph of the greedy pseudoline arrangement is sorted (\ie of the form~$(i,j)$ with~$i<j$)~---~see \fref{fig:contact} (left).
\end{example}

\begin{example}[Duplicated networks, continued]\label{exm:duplicated2}
For a connected graph $\Gamma$, consider the duplicated network $\cZ(\Gamma)$ defined in Example~\ref{exm:duplicated1}. A pseudoline arrangement supported by~$\cZ(\Gamma)$ has one contact among each pair of duplicated commutators, and thus its contact graph is an oriented copy of~$\Gamma$. Reciprocally, any orientation on~$\Gamma$ is the contact graph of a pseudoline arrangement: this follows from Lemma~\ref{lem:contactgraph} since~$\cZ(\Gamma)$ supports~$2^{|\Gamma|}$ pseudoline arrangements, but one can also easily reconstruct the pseudoline arrangement whose contact graph is a given orientation on~$\Gamma$.

Note that the contact graphs of the pseudoline arrangements supported by a given network have in general distinct underlying undirected graphs (see  \fref{fig:contact}).
\end{example}

\begin{remark}
In fact, any labeled directed multigraph arises as the contact graph of pseudoline arrangement on a certain sorting network. Indeed, consider a labeled directed multigraph~$G$ on $n$ vertices. Consider the unique pseudoline arrangement~$\Lambda$ supported by a reduced network~$\cN$ with $n$ levels and ${n \choose 2}$ commutators. For any directed edge ${(i,j)\in G}$ with $i<j$ (resp.~with $i>j$), insert a new commutator immediately to the right (resp.~left) of the crossing between the $i$\textsuperscript{th} and $j$\textsuperscript{th} pseudolines of $\Lambda$. Then $G$ is precisely the contact graph $\Lambda^\contact$ of $\Lambda$ (seen as a pseudoline arrangement supported by the resulting network).
\end{remark}

Let $\Lambda$ and $\Lambda'$ denote two pseudoline arrangements supported by~$\cN$ and related by a flip involving their $i$\textsuperscript{th} and $j$\textsuperscript{th} pseudolines~---~see \fref{fig:contact} for an example. Then the directed multigraphs obtained by merging the vertices $i$ and $j$ in the contact graphs $\Lambda^\contact$ and $\Lambda'^\contact$ coincide. In particular, a flip preserves the connected components of the contact graph. Since the flip graph~$G(\cN)$ is connected (Theorem~\ref{theo:connect}), this implies the following result:

\begin{lemma}\label{lem:connectcomp}
The contact graphs of all pseudoline arrangements supported by~$\cN$ have the same connected components.
\end{lemma}

We call a sorting network \defn{reducible} (resp.~\defn{irreducible}) when the contact graphs of the pseudoline arrangements it supports are disconnected (resp. connected).

Our next statement describes the structure of the simplicial complex~$\Delta(\cN)$ and of the brick polytope~$\Omega(\cN)$ associated to a reducible sorting network~$\cN$. To formalize it, we need the following definition. If $\cN$ is a network and $\Theta$ is a set of disjoint abscissa monotone curves supported by~$\cN$, the \defn{restriction} of~$\cN$ to~$\Theta$ is the network obtained by keeping only the curves of $\Theta$ and the commutators between them, and by stretching the curves of $\Theta$. In other words, it has $|\Theta|$ levels and a commutator between its $i$\textsuperscript{th} and $(i+1)$\textsuperscript{th} levels for each commutator of $\cN$ joining the $i$\textsuperscript{th} and $(i+1)$\textsuperscript{th} curves of $\Theta$~---~see \fref{fig:restriction} (left). 

\begin{figure}
  \capstart
  \centerline{\includegraphics[width=.9\textwidth]{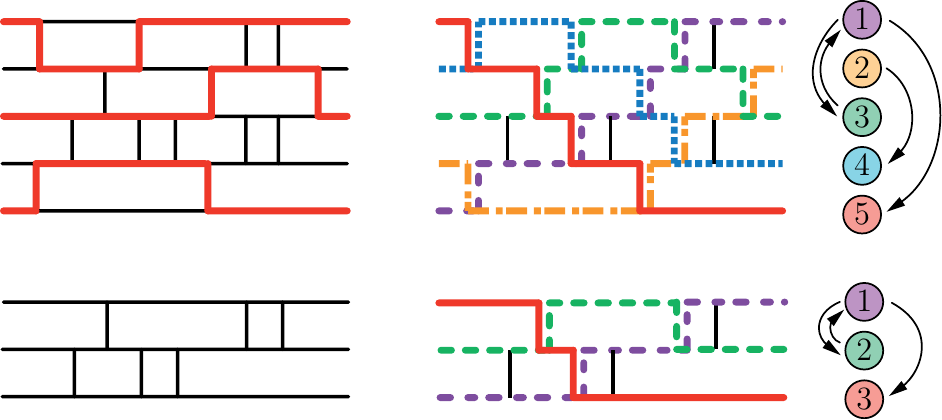}}
  \caption{A sorting network and a pseudoline arrangement covering it (top). Their restriction to the connected component $\{1,3,5\}$ of the contact graph (bottom).}
  \label{fig:restriction}
\end{figure}

It makes sense to speak of the restriction~$\cN(U)$ of~$\cN$ to a connected component~$U$ of the contact graphs of the pseudoline arrangements supported by~$\cN$. Indeed, if~$\Lambda$ is supported by~$\cN$, the restriction of~$\cN$ to the levels of the subarrangement formed by the pseudolines of~$\Lambda$ labeled by~$U$ does not depend on the choice of~$\Lambda$~---~see \fref{fig:restriction} (right). Furthermore, there is an obvious correspondence between the pseudoline arrangements supported by~$\cN(U)$ and the subarrangements of the arrangements supported by~$\cN$ formed by their pseudolines in~$U$. In particular,~$\cN(U)$ is an irreducible sorting network; we say that it is an \defn{irreducible component} of~$\cN$.

\begin{proposition}\label{prop:disconnected}
Let $\cN$ be a sorting network whose irreducible components are $\cN_1,\dots,\cN_p$. Then the simplicial complex~$\Delta(\cN)$ is isomorphic to the join of the simplicial complexes~$\Delta(\cN_1),\dots,\Delta(\cN_p)$ and the brick polytope~$\Omega(\cN)$ is a translate of the product of the brick polytopes~$\Omega(\cN_1),\dots,\Omega(\cN_p)$.
\end{proposition}

\begin{proof}
The commutators of $\cN$ can be partitioned into~$p+1$ sets: one set corresponding to each irreducible component of~$\cN$, and the set~$X$ of commutators between two different connected components in the contact graphs. All pseudoline arrangements supported by~$\cN$ have crossings at the commutators of~$X$, and are obtained by choosing independently their subarrangements on the irreducible components~$\cN_1,\dots,\cN_p$. The result immediately follows.
\end{proof}

In particular, the dimension of the brick polytope of a sorting network with $n$ levels and $p$ irreducible components is at most $n-p$. We will see in Corollary~\ref{coro:dim} that this is the exact dimension. For example, the brick polytope of a reduced network (see Example~\ref{exm:reduced}) has dimension $0$: it supports a unique pseudoline arrangement whose contact graph has no edge, and its brick polytope is a single point.

Proposition~\ref{prop:disconnected} enables us to only focus on irreducible sorting networks throughout this article. Among them, the following networks have the fewest commutators:

\begin{definition}\label{def:minimal}
An irreducible sorting network~$\cN$ is \defn{minimal} if it satisfies the following equivalent conditions:
\begin{enumerate}[(i)]
\item $\cN$ has $n$ levels and $m={n \choose 2}+n-1$ commutators.
\item The contact graph of a pseudoline arrangement supported by~$\cN$ is a tree.
\item The contact graphs of all pseudoline arrangements supported by~$\cN$ are trees.
\end{enumerate}
\end{definition}

For example, the networks of \fref{fig:network} all have $5$ levels and $14$ commutators. The rightmost is reducible, but the other two are minimal. To be convinced, draw the greedy pseudoline arrangement on these networks, and check that its contact graph is connected.

We come back to minimal irreducible sorting networks at the end of Section~\ref{subsec:fd} since their brick polytopes are of particular interest.

\subsection{The incidence cone of a directed multigraph}

In this section, we briefly recall classical properties of the vector configuration formed by the columns of the incidence matrix of a directed multigraph. We fix a directed multigraph~$G$ on~$n$ vertices, whose underlying undirected graph is connected. 
Let~$(e_1,\dots,e_n)$ be the canonical basis of~$\R^n$ and $\one\eqdef\sum e_i$.

\begin{definition}
The \defn{incidence configuration} of the directed multigraph~$G$ is the vector configuration~$I(G) \eqdef \set{e_j-e_i}{(i,j)\in G} \subset\R^n$. The \defn{incidence cone} of~$G$ is the cone~$C(G)\subset\R^n$ generated by~$I(G)$, {\it i.e.} its positive span.
\end{definition}

In other words, the incidence configuration of a directed multigraph consists of the column vectors of its incidence matrix. Observe that the incidence cone is contained in the linear subspace of equation $\dotprod{\one}{x} = 0$. We will use the following relations between the graph properties of~$G$ and the orientation properties of~$I(G)$, which can be summed up by saying that the (oriented and unoriented) matroid of $G$ coincides with that of its incidence configuration $I(G)$. See~\cite{bvswz} for an introduction and reference on oriented matroids. In particular, Section 1.1 of that book explores the incidence configuration of a directed graph.

\begin{remark}\label{rem:incidenceconfiguration}
Consider a subgraph~$H$ of~$G$. Then the vectors of~$I(H)$:
\begin{enumerate}
\item are independent if and only if $H$~has no (not necessarily oriented) cycle, that is, if $H$ is a forest;
\item span the hyperplane~$\dotprod{\one}{x}= 0$ if and only if $H$~is connected and spanning;
\item form a basis of the hyperplane~$\dotprod{\one}{x}= 0$ if and only if $H$~is a spanning~tree;
\item form a circuit if and only if $H$~is a (not necessarily oriented) cycle; the positive and negative parts of the circuit correspond to the subsets of edges oriented in one or the other direction along this cycle; in particular,~$I(H)$ is a positive circuit if and only if $H$~is an oriented cycle;
\item form a cocircuit if and only if $H$~is a minimal (not necessarily oriented) cut; the positive and negative parts of the cocircuit correspond to the edges in one or the other direction in this cut; in particular,~$I(H)$ is a positive cocircuit if and only if $H$~is an oriented cut.
\end{enumerate}
\end{remark}

This remark on the incidence configuration translates into the following remark on the incidence cone:

\begin{remark}\label{rem:incidencecone}
Consider a subgraph~$H$ of~$G$. The incidence configuration $I(H)$ is the set of vectors of $I(G)$ contained in a $k$-face of $C(G)$ if and only if $H$ has $n-k$ connected components and the quotient graph $G/H$ is acyclic. In particular:
\begin{enumerate}
\item The cone~$C(G)$ has dimension~$n-1$ (since we assumed that the undirected graph underlying~$G$ is connected).
\item The cone~$C(G)$ is pointed if and only if~$G$ is an acyclic directed graph.
\item If~$G$ is acyclic, it induces a partial order on its set of nodes. The rays of~$C(G)$ correspond to the edges of the Hasse diagram of~$G$. The cone is simple if and only if the Hasse diagram of~$G$ is a tree.
\item The facets of~$C(G)$ correspond to the complements of the minimal directed cuts in~$G$. Given a minimal directed cut in~$G$, the characteristic vector of its sink is a normal vector of the corresponding facet.
\end{enumerate}
\end{remark}

\begin{example}
If~$G$ is the complete directed graph on~$n$ vertices, with one arc from any node to any other (and thus, two arcs between any pair of nodes, one in each direction), then its incidence configuration~$I(G) = \set{e_i-e_j}{i,j \in [n]}$ is the root system of type~$A$. See~\cite{Humphreys}.

If~$G$ is an acyclic orientation on the complete graph, then its incidence configuration~$I(G)$ is a positive system of roots~\cite{Humphreys}. The Hasse diagram of the order induced by~$G$ is a path, thus the incidence cone~$C(G)$ is simple. Its rays are spanned by the simple roots of the positive system~$I(G)$.
\end{example}

\subsection{The vertices of the brick polytope}\label{subsec:vc}

Let~$\cN$ be an irreducible sorting network supporting a pseudoline arrangement~$\Lambda$. We use the contact graph~$\Lambda^\contact$ to describe the cone of the brick polytope~$\Omega(\cN)$ at the brick vector~$\omega(\Lambda)$:

\begin{theorem}\label{theo:cones}
The cone of the brick polytope $\Omega(\cN)$ at the brick vector $\omega(\Lambda)$ is precisely the incidence cone $C(\Lambda^\contact)$ of the contact graph~$\Lambda^\contact$ of $\Lambda$:
$$\cone \set{\omega(\Lambda')-\omega(\Lambda)}{\Lambda' \in \arr(\cN)} = \cone \set{e_j-e_i}{(i,j) \in \Lambda^\contact}.$$
\end{theorem}

\begin{proof}
Assume that $\Lambda'$ is obtained from $\Lambda$ by flipping a contact from its $i$\textsuperscript{th} pseudoline to its $j$\textsuperscript{th} pseudoline. Then the difference $\omega(\Lambda')-\omega(\Lambda)$ is a positive multiple of $e_j-e_i$. This immediately implies that the incidence cone $C(\Lambda^\contact)$ is included in the cone of $\Omega(\cN)$ at $\omega(\Lambda)$.

Reciprocally, we have to prove that any facet $F$ of the cone $C(\Lambda^\contact)$ is also a facet of the brick polytope $\Omega(\cN)$. 
According to Remark \ref{rem:incidencecone}(4), there exists a minimal directed cut from a source set $U$ to a sink set $V$ (which partition the vertices of $\Lambda^\contact$) such that $\one_V \eqdef \sum_{v \in V} e_v$ is a normal vector of $F$. We denote by $\gamma$ the commutators of $\cN$ which correspond to the arcs of $\Lambda^\contact$ between~$U$ and~$V$. We claim that for any pseudoline arrangement $\Lambda'$ supported by $\cN$, the scalar product $\dotprod{\one_V}{\omega(\Lambda')}$ equals $\dotprod{\one_V}{\omega(\Lambda)}$ when $\gamma$ is a subset of the contacts of $\Lambda'$, and is strictly bigger than $\dotprod{\one_V}{\omega(\Lambda)}$ otherwise.

Remember first that the set of all pseudoline arrangements supported by $\cN$ and whose set of contacts contains $\gamma$ is connected by flips. Since a flip between two such pseudoline arrangements necessarily involves either two pseudolines of $U$ or two pseudolines of $V$, the corresponding incidence vector is orthogonal to $\one_V$. Thus, the scalar product $\dotprod{\one_V}{\omega(\Lambda')}$ is constant on all pseudoline arrangements whose set of contacts contains $\gamma$.

Reciprocally, we consider a pseudoline arrangement $\Lambda'$ supported by $\cN$ which minimizes the scalar product $\dotprod{\one_V}{\omega(\Lambda')}$. There is clearly no arc from $U$ to $V$ in~$\Lambda'^\contact$, otherwise flipping the corresponding contact in $\Lambda'$ would decrease the value of $\dotprod{\one_V}{\omega(\Lambda')}$. We next prove that we can join $\Lambda$ to $\Lambda'$ by flips involving two pseudolines of $U$ or two pseudolines of $V$. As a first step, we show that we can transform $\Lambda$ and $\Lambda'$ into pseudoline arrangements $\hat\Lambda$ and $\hat \Lambda'$ in which the first pseudoline coincide, using only flips involving two pseudolines of $U$ or two pseudolines of $V$. We can then conclude by induction on the number of levels of~$\cN$.

Assume first that the first pseudoline (the one which starts at level $1$ and ends at level $n$) of $\Lambda$ and $\Lambda'$ is in $U$. We sweep this pseudoline from left to right in $\Lambda$. If there is a contact above and incident to it, the above pseudoline must be in $U$. Otherwise we would have an arc between $V$ and $U$ in $\Lambda^\contact$. Consequently, we are allowed to flip this contact. By doing this again and again we obtain a pseudoline arrangement $\hat\Lambda$ whose first pseudoline starts at the bottom leftmost point and goes up whenever possible until getting to the topmost level. Since this procedures only relies on the absence of arc from $V$ to $U$ in $\Lambda^\contact$, we can proceed identically on $\Lambda'$ to get a pseudoline arrangement $\hat\Lambda'$ with the same first pseudoline. Finally, if the first pseudoline of $\Lambda$ and $\Lambda'$ is in~$V$, then we can argue similarly but sweeping the pseudoline from right to left.
\end{proof}

By Remark~\ref{rem:incidencecone}(1) and Proposition~\ref{prop:disconnected}, we obtain the dimension of~$\Omega(\cN)$:

\begin{corollary}\label{coro:dim}
The brick polytope of an irreducible sorting network with $n$ levels has dimension $n-1$. In general, the brick polytope of a sorting network with $n$ levels and $p$ irreducible components has dimension $n-p$.
\end{corollary}

According to Remark \ref{rem:incidencecone}(2), Theorem~\ref{theo:cones} also characterizes the pseudoline arrangements whose brick vector is a vertex of $\Omega(\cN)$:

\begin{corollary}\label{coro:vc}
The brick vector $\omega(\Lambda)$ is a vertex of the brick polytope $\Omega(\cN)$ if and only if the contact graph $\Lambda^\contact$ of $\Lambda$ is acyclic.
\end{corollary}

For example, the brick vector of the greedy pseudoline arrangement~$\Gamma(\cN)$ is always a vertex of~$\Omega(\cN)$ since its contact graph is sorted (see Example~\ref{exm:contactgreedy}). Similarly, the brick vector of the sink of the oriented graph of decreasing flips is always a vertex of~$\Omega(\cN)$. These two greedy pseudoline arrangements can be the only vertices of the brick polytope, as happens for $2$-level networks:

\begin{example}[$2$-level networks, continued]
Let $\cX_m$ be the sorting network formed by two levels related by $m$ commutators. The contact graph of the pseudoline arrangement whose unique crossing is the $i$\textsuperscript{th} commutator of~$\cX_m$ is a multigraph with two vertices and $m-1$ edges, $m-i$ of them in one direction and $i-1$ in the other. Thus, only the first and last commutators give pseudoline arrangements with acyclic contact graphs.
\end{example}

In general, the map~$\omega:\arr(\cN)\to\R^n$ (which associates to a pseudoline arrangement its brick vector) is not injective on~$\arr(\cN)$. For example, many interior points appear several times in Examples~\ref{exm:3levels} and~\ref{exm:duplicated1}.
However, the vertices of the brick polytope have precisely one preimage by~$\omega$:

\begin{proposition}\label{prop:vertices}
The map~$\omega:\arr(\cN)\to\R^n$ restricts to a bijection between the pseudoline arrangements supported by~$\cN$ whose contact graphs are acyclic and the vertices of the brick polytope~$\Omega(\cN)$.
\end{proposition}

\begin{proof}
According to Corollary~\ref{coro:vc}, the map~$\omega$ defines a surjection from the pseudoline arrangements supported by~$\cN$ whose contact graphs are acyclic to the vertices of the brick polytope~$\Omega(\cN)$. To prove injectivity, we use an inductive argument based on the following claims:
\begin{enumerate}[(i)]
\item the greedy pseudoline arrangement~$\Gamma(\cN)$ is the unique preimage of~$\omega(\Gamma(\cN))$;
\item if a vertex of $\Omega(\cN)$ has a unique preimage by~$\omega$, then so do its neighbors in the graph of~$\Omega(\cN)$.
\end{enumerate}

To prove~(i), consider a pseudoline arrangement~$\Lambda$ supported by~$\cN$ such that $\omega(\Lambda)=\omega(\Gamma(\cN))$. According to Theorem~\ref{theo:cones}, the contact graphs~$\Lambda^\contact$ and~$\Gamma(\cN)^\contact$ have the same incidence cone, which ensures that all arcs of~$\Lambda^\contact$ are sorted. In other words, all flips in~$\Lambda$ are decreasing. Since this property characterizes the greedy pseudoline arrangement, we obtain that~$\Lambda=\Gamma(\cN)$.

To prove~(ii), consider two neighbors $v,v'$ in the graph of~$\Omega(\cN)$. Let ${i,j\in[n]}$ be such that $v'-v=\alpha(e_j-e_i)$ for some $\alpha>0$. Let $\Lambda$ be a pseudoline arrangement supported by~$\cN$ such that $v=\omega(\Lambda)$. Let $\Lambda'$ denote the pseudoline arrangement obtained from $\Lambda$ by flipping the rightmost contact between its $i$\textsuperscript{th} and $j$\textsuperscript{th} pseudolines if $i<j$ and the leftmost one if $i>j$. Then $v'=\omega(\Lambda')$. In particular, if $v$ has two distinct preimages by $\omega$, then so does $v'$. This proves~(ii).
\end{proof}

\begin{corollary}\label{coro:graph}
The graph of the brick polytope is a subgraph of~$G(\cN)$ whose vertices are the pseudoline arrangements with acyclic contact graphs.
\end{corollary}

Example~\ref{exm:counterExample} shows that the graph of the brick polytope is not always the subgraph of~$G(\cN)$ induced by the pseudoline arrangements with acyclic~contact~graphs.

\begin{example}[Duplicated networks, continued]\label{exm:duplicated3}
For a connected graph~$\Gamma$, consider the duplicated network~$\cZ(\Gamma)$ defined in Example~\ref{exm:duplicated1}. The contact graphs of the~$2^{|\Gamma|}$ pseudoline arrangements supported by~$\cZ(\Gamma)$ are the $2^{|\Gamma|}$ orientations on~$\Gamma$, and two pseudoline arrangements are related by a flip if their contact graphs differ in the orientation of a single edge of~$\Gamma$. According to Proposition~\ref{prop:vertices}, the vertices of the brick polytope~$\Omega(\cZ(\Gamma))$ correspond to the acyclic orientations on~$\Gamma$.

\begin{figure}[b]
  \capstart
  \centerline{\includegraphics[scale=.6]{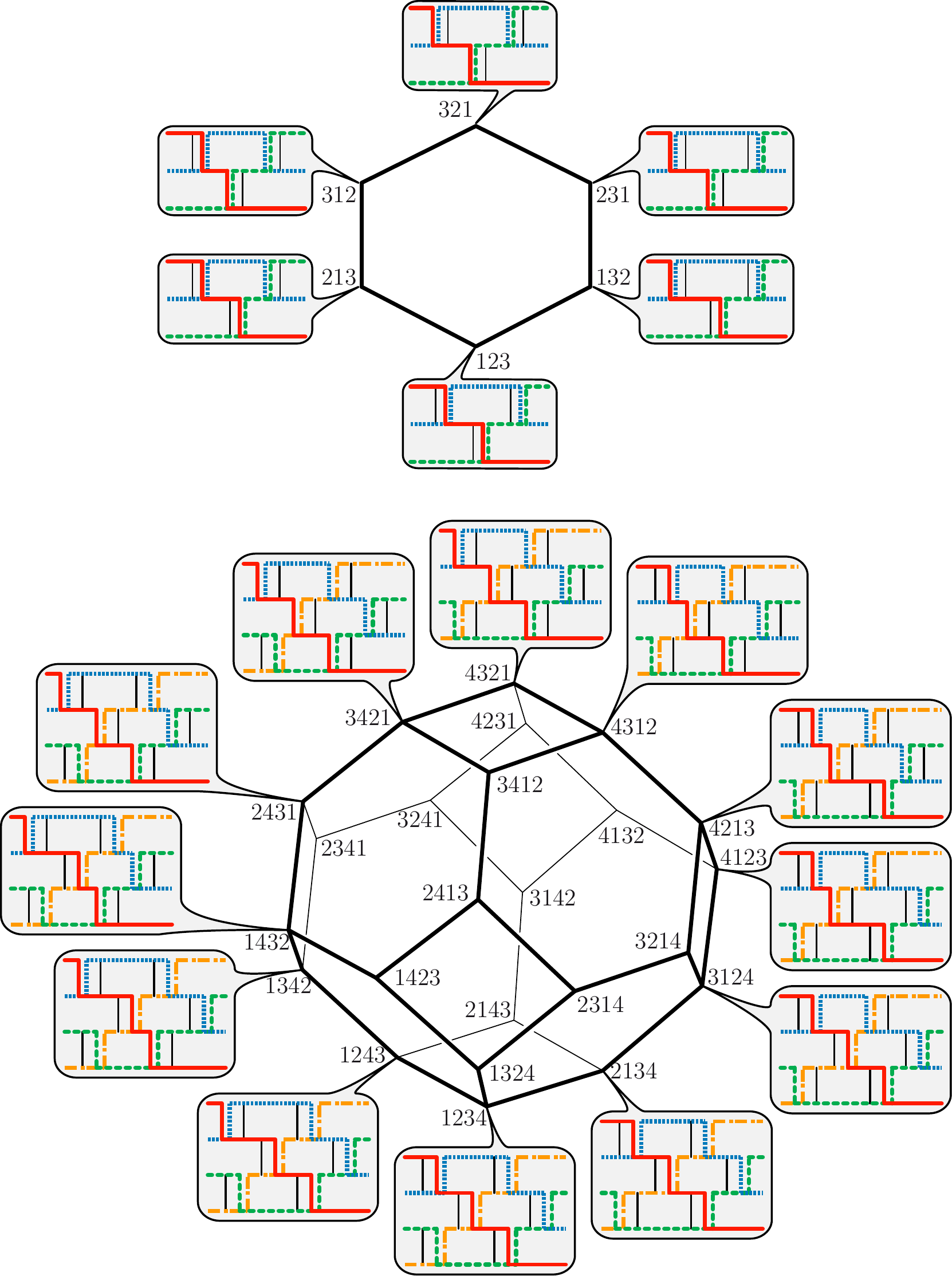}}
  \caption{The brick polytope~$\Omega(\cZ(K_n))$ is (a translate of) the permutahedron~$\Pi_n$. For space reason, we have represented only certain  arrangements. The rest of the drawing is left to the reader.}
  \label{fig:permutahedron}
\end{figure}

When~$\Gamma=K_n$ is the complete graph on $n$ vertices, the contact graphs of the pseudoline arrangements supported by $\cZ(K_n)$ are the tournaments on $[n]$, the vertices of $\Omega(\cZ(K_n))$ correspond to the permutations of~$[n]$, and the graph of $\Omega(\cZ(K_n))$ is a subgraph of that of the permutahedron~$\Pi_n=\conv\set{(\sigma(1),\dots,\sigma(n))^T}{\sigma\in\mathfrak{S}_n}$. Since $\Pi_n$ is simple and both $\Omega(\cZ(K_n))$ and $\Pi_n$ have dimension~$n-1$, they must have in fact the same graph, and consequently the same combinatorial structure (by simplicity~\cite{BlindMani, Kalai}). In fact, $\Omega(\cZ(K_n))$ is a translate of~$\Pi_n$. See \fref{fig:permutahedron}.

When $\Gamma$ is a tree, all possible orientations on $\Gamma$ are acyclic. The brick polytope $\Omega(\cZ(\Gamma))$ is thus a cube.
\end{example}

\begin{remark}[Simple and non-simple vertices]
According to Remark~\ref{rem:incidencecone}(3), the brick vector $\omega(\Lambda)$ of a pseudoline arrangement~$\Lambda$ supported by~$\cN$ is a simple vertex of~$\Omega(\cN)$ if and only if the contact graph~$\Lambda^\contact$ is acyclic and its Hasse diagram is a tree. See \fref{fig:counterExample} for a brick polytope with non-simple vertices.
\end{remark}

\subsection{The faces of the brick polytope}\label{subsec:fd}

Let $\cN$ be an irreducible sorting network. Theorem~\ref{theo:cones} and Remark~\ref{rem:incidencecone}(4) provide the facet description of the brick polytope~$\Omega(\cN)$:

\begin{corollary}\label{coro:fd}
The facet normal vectors of the brick polytope $\Omega(\cN)$ are precisely all facet normal vectors of the incidence cones of the contact graphs of the pseudoline arrangements supported by $\cN$. Representatives for them are given by the characteristic vectors of the sinks of the minimal directed cuts of these~contact~graphs.
\end{corollary}

\begin{remark}
Since we know its vertices and its facet normal vectors, we obtain immediately the complete inequality description of the brick polytope. More precisely, for each given pseudoline arrangement~$\Lambda$ supported by~$\cN$ with an acyclic contact graph~$\Lambda^\contact$, and for each minimal directed cut in~$\Lambda^\contact$ with source~$U$ and sink~$V$, the right-hand-side of the inequality of the facet with normal vector~$\one_V \eqdef \sum_{v \in V} e_v$ is given by the sum, over all pseudolines~$\ell$ of~$\Lambda$ in~$V$, of the number of bricks below~$\ell$.
\end{remark}

More generally, Theorem~\ref{theo:cones} implies a combinatorial description of the faces of the brick polytope. We need the following definition:

\begin{definition}
A set $\gamma$ of commutators of~$\cN$ is \defn{$k$-admissible} if there exists a pseudoline arrangement $\Lambda\in\arr(\cN|\gamma)$ such that $\Lambda^\contact\ssm\gamma^\contact$ has $n-k$ connected components and $\Lambda^\contact/(\Lambda^\contact\ssm\gamma^\contact)$ is acyclic (where $\gamma^\contact$ denotes the subgraph of~$\Lambda^\contact$ corresponding to the commutators of~$\gamma$).
\end{definition}

Theorem~\ref{theo:cones}, Remark~\ref{rem:incidencecone}, and Proposition~\ref{prop:vertices} lead to our face description:

\begin{corollary}\label{coro:faces}
Let $\Phi$ be the map which associates to a subset $X$ of $\R^n$ the set of commutators of $\cN$ which are contacts in all the pseudoline arrangements supported by~$\cN$ whose brick vectors lie in~$X$. Let $\Psi$ be the map which associates to a set~$\gamma$ of commutators of~$\cN$ the convex hull of $\set{\omega(\Lambda)}{\Lambda\in\arr(\cN|\gamma)}$. Then the maps $\Phi$ and $\Psi$ define inverse bijections between the $k$-faces of~$\Omega(\cN)$ and the $k$-admissible sets of commutators of~$\cN$.
\end{corollary}

\begin{example}[Duplicated networks, continued]\label{exm:duplicated4}
For a connected graph~$\Gamma$ on~$[n]$, consider the duplicated network~$\cZ(\Gamma)$ defined in Example~\ref{exm:duplicated1}. According to Corollary~\ref{coro:faces}, the $k$-faces of $\Omega(\cZ(\Gamma))$ are in bijection with the couples $(\pi,\sigma)$ where $\pi$ is a set of~$n-k$ connected induced subgraphs of~$\Gamma$ whose vertex sets partition~$[n]$, and $\sigma$ is an acyclic orientation on the quotient $\Gamma/\bigcup \pi$.

When $\Gamma=K_n$ is the complete graph on $n$ vertices, the $k$-faces of the brick polytope $\Omega(\cZ(K_n))$ correspond to the ordered $(n-k)$-partitions of~$[n]$. This confirms that $\Omega(\cZ(K_n))$ is combinatorially equivalent to the permutahedron $\Pi_n$.

When~$\Gamma$ is a tree, we obtain a $k$-face by choosing $n-k$ connected subgraphs of~$\Gamma$ whose vertex sets partition~$[n]$ and an arbitrary orientation on the remaining edges. This clearly corresponds to the $k$-faces of the cube.
\end{example}

We now apply our results to minimal irreducible sorting networks (which support pseudoline arrangements whose contact graphs are trees~---~Definition~\ref{def:minimal}). For these networks, we obtain an affirmative answer to Question~\ref{qu:polytope}:

\begin{theorem}
For any minimal irreducible sorting network~$\cN$, the simplicial complex~$\Delta(\cN)$ is the boundary complex of the polar of the brick polytope~$\Omega(\cN)$. In particular, the graph of $\Omega(\cN)$ is the flip graph $G(\cN)$.
\end{theorem}

\begin{proof}
Let~$\gamma$ be a set of $p$~commutators such that $\cN\ssm\gamma$ is still sorting, and let $\Lambda\in\arr(\cN|\gamma)$. Since the contact graph of $\Lambda$ is an oriented tree, its subgraph $\Lambda^\contact\ssm\gamma^\contact$ has $p+1$ connected components and its quotient $\Lambda^\contact/(\Lambda^\contact\ssm\gamma^\contact)$ is acyclic. Consequently, $\Omega(\cN)$ has a $(n-p-1)$-dimensional face corresponding to~$\gamma$.
\end{proof}

\begin{example}
When~$\Gamma$ is a tree, the polar of the cube~$\Omega(\cZ(\Gamma))$ realizes~$\Delta(\cZ(\Gamma))$.
\end{example}

Reciprocally, note that the dimension of the brick polytope~$\Omega(\cN)$ does not even match that of the simplicial complex~$\Delta(\cN)$ for an irreducible sorting network~$\cN$ which is not minimal. Consequently, the minimal networks are precisely those irreducible networks for which the brick polytope provides a realization of $\Delta(\cN)$. We discuss in more details the relationship between the boundary complex of the brick polytope and the simplicial complex $\Delta(\cN)$ for certain non minimal networks in Section~\ref{sec:multi}.

\begin{example}\label{exm:counterExample}
We finish our combinatorial description of the face structure of the brick polytope with the example of the sorting network represented in \fref{fig:counterExample}. It illustrates that:
\begin{enumerate}
\item All pseudoline arrangements supported by~$\cN$ can appear as vertices of~$\Omega(\cN)$ even for a non-minimal network~$\cN$.
\item Even when all pseudoline arrangements supported by~$\cN$ appear as vertices of~$\Omega(\cN)$, the graph can be a strict subgraph of $G(\cN)$.
\item The brick vector of a pseudoline arrangement~$\Lambda$ supported by~$\cN$ is a simple vertex of~$\Omega(\cN)$ if and only if the contact graph~$\Lambda^\contact$ is acyclic and its Hasse diagram is a tree.
\end{enumerate}

\begin{figure}
  \capstart
  \centerline{\includegraphics[width=\textwidth]{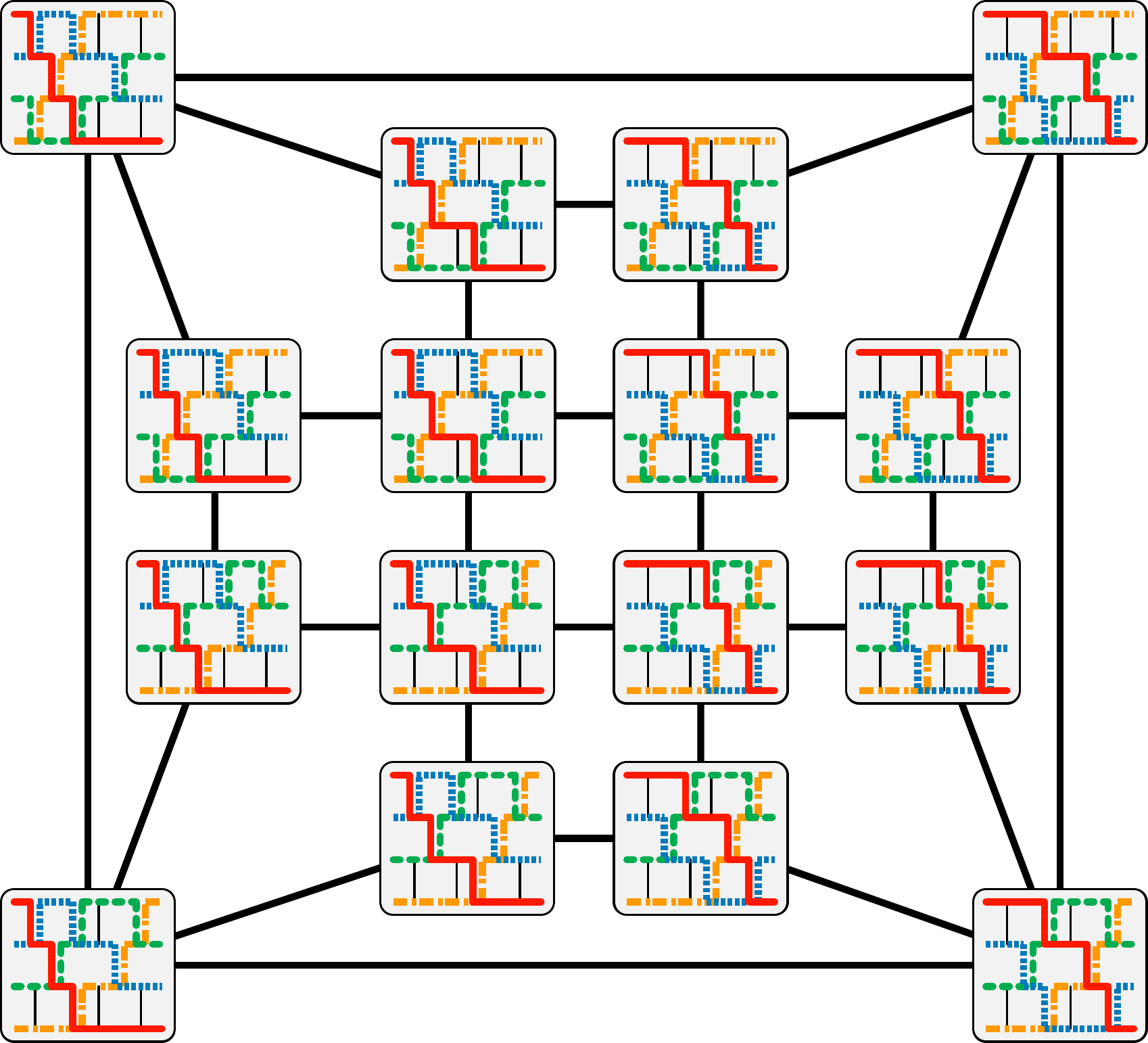}}
  \caption{The Schlegel diagram of the brick polytope~$\Omega(\cN)$ of a sorting network~$\cN$ with $4$ levels and $10$ commutators. Since the graph of flips~$G(\cN)$ is that of the $4$-dimensional cube, $\Omega(\cN)$ has no missing vertex but $4$ missing edges (find them!).}
  \label{fig:counterExample}
\end{figure}
\end{example}

\subsection{Brick polytopes as (positive) Minkowski sums}\label{subsec:minkowskiSum}

Let~$\cN$ be a sorting network with $n$ levels and let~$b$ be a brick of~$\cN$. For any pseudoline arrangement~$\Lambda$ supported by~$\cN$, we denote by~$\omega(\Lambda,b)\in\R^n$  the characteristic vector of the pseudolines of~$\Lambda$ passing above~$b$. We associate to the brick~$b$ of ~$\cN$ the polytope
$$\Omega(\cN,b) \eqdef \conv\set{\omega(\Lambda,b)}{\Lambda \in \arr(\cN)}\subset\R^n.$$
These polytopes provide a Minkowski sum decomposition of the brick polytope~$\Omega(\cN)$:

\begin{proposition}\label{prop:minkowskiSum}
The brick polytope~$\Omega(\cN)$ is the Minkowski sum of the polytopes~$\Omega(\cN,b)$ associated to all the bricks $b$ of~$\cN$.
\end{proposition}

\begin{proof}
Since $\omega(\Lambda) = \sum \omega(\Lambda,b)$ for any pseudoline arrangement~$\Lambda$ supported by~$\cN$, $\Omega(\cN)$ is included in the Minkowski sum~$\sum \Omega(\cN,b)$. To prove equality, we thus only have to prove that any vertex of~$\sum \Omega(\cN,b)$ is also a vertex of~$\Omega(\cN)$. 

Let $f:\R^n\to\R$ be a linear function, and $\Lambda, \Lambda'$ be two pseudoline arrangements related by a flip involving their $i$\textsuperscript{th} and $j$\textsuperscript{th} pseudolines. If a brick~$b$ of~$\cN$ is not located between the $i$\textsuperscript{th} pseudolines of~$\Lambda$ and~$\Lambda'$, then $f(\omega(\Lambda,b)) = f(\omega(\Lambda',b))$. Otherwise, the variation $f(\omega(\Lambda,b)) - f(\omega(\Lambda',b))$ has the same sign as the variation $f(\omega(\Lambda)) - f(\omega(\Lambda'))$. Consequently, the pseudoline arrangement~$\Lambda_f$ supported by~$\cN$ which minimizes $f$ on $\Omega(\cN)$, also minimizes $f$ on $\Omega(\cN,b)$ for each brick~$b$ of~$\cN$.

Now let $v$ be any vertex of~$\sum \Omega(\cN,b)$. Let $f:\R^n\to\R$ denote a linear function which is minimized by~$v$ on $\sum \Omega(\cN,b)$. Then $v$ is the sum of the vertices which minimize~$f$ in each summand $\Omega(\cN,b)$. Consequently, we obtain that $v=\sum\omega(\Lambda_f,b)=\omega(\Lambda_f)$ is a vertex of~$\Omega(\cN)$.
\end{proof}

\begin{remark}
This Minkowski sum decomposition comes from the fact that the summation and convex hull in the definition of the brick polytope commute:
$$\Omega(\cN) \eqdef \conv_\Lambda \sum\nolimits_b \omega(\Lambda,b) \,=\, \sum\nolimits_b \conv_\Lambda \omega(\Lambda,b) \eqfed \sum\nolimits_b \Omega(\cN,b),$$
where the index~$b$ of the sums ranges over the bricks of~$\cN$ and the index~$\Lambda$ of the convex hulls ranges over the pseudoline arrangements supported by~$\cN$.
\end{remark}

Observe that the vertex set of~$\Omega(\cN,b)$ is contained in the vertex set of a hypersimplex, since the number of pseudolines above~$b$ always equals the depth of~$b$. Furthermore, since~$\Omega(\cN,b)$ is a summand in a Minkowski decomposition of~$\Omega(\cN)$, all the edges of the former are parallel to that of the latter, and thus to that of the standard simplex $\simplex_{[n]} \eqdef \conv\set{e_i}{i\in [n]}$. Consequently, the polytope~$\Omega(\cN,b)$ is a \defn{matroid polytope}, \ie the convex set of the characteristic vectors of the bases of a matroid (the ground set of this matroid is the set of pseudolines and its rank is the depth of~$b$). 

\begin{example}[$2$-level networks, continued]
Let $\cX_m$ be the sorting network formed by two levels and $m$ commutators, and let~$b$ be a brick of~$\cX_m$. For every pseudoline arrangement~$\Lambda$ supported by~$\cX_m$, we have $\omega(\Lambda,b)=(1,0)$ if the two pseudolines of~$\Lambda$ cross before~$b$ and $\omega(\Lambda,b)=(0,1)$ otherwise. Thus, the polytope $\Omega(\cX_m,b)$ is the segment with endpoints $(1,0)$ and $(0,1)$, and the brick polytope~$\Omega(\cX_m)$ is the Minkowski sum of $m-1$ such segments.
\end{example}

\begin{example}[Duplicated networks, continued]\label{exm:duplicated5}
For a connected graph~$\Gamma$, consider the duplicated network~$\cZ(\Gamma)$ obtained by duplicating the commutators of a reduced network~$\cN$ according to the edges of~$\Gamma$~---~see Example~\ref{exm:duplicated1}. This network has two kinds of bricks: those located between two adjacent commutators (which replace a commutator of~$\cN$) and the other ones (which correspond to the bricks of~$\cN$). For any brick~$b$ of the latter type, the polytope $\Omega(\cZ(\Gamma),b)$ is still a single point. Now let~$b$ be a brick of~$\cZ(\Gamma)$ located between a pair of adjacent commutators corresponding to the edge $\{i,j\}$ of~$\Gamma$. Then $\Omega(\cZ(\Gamma),b)$ is (a translate of) the segment~$[e_i,e_j]$. Summing the contributions of all bricks, we obtain that the brick polytope~$\Omega(\cZ(\Gamma))$ is the Minkowski sum of all segments $[e_i,e_j]$ for $\{i,j\}\in \Gamma$. Such a polytope is usually called a \defn{graphical zonotope}. When $\Gamma = K_n$, the permutahedron $\Pi_n = \Omega(\cZ(K_n))$ is the Minkowski sum of the segments $[e_i,e_j]$ for all distinct $i,j\in[n]$. When $\Gamma$ is a tree, the cube $\Omega(\cZ(\Gamma))$ is the Minkowski sum of linearly independent segments.
\end{example}

\subsection{Brick polytopes and generalized permutahedra}\label{subsec:generalizedPermutahedra}

Our brick polytopes are instances of a well-behaved class of polytopes studied in~\cite{Postnikov, ArdilaBenedettiDoker, PostnikovReinerWilliams}:

\begin{definition}[\cite{Postnikov}]
A \defn{generalized permutahedra} is a polytope whose inequality description is of the form:
$$\GP{Z}{z_I} \eqdef \set{\begin{pmatrix} x_1 \\ \vdots \\ x_n\end{pmatrix}\in\R^n}{\sum_{i=1}^n x_i = z_{[n]} \text{ and } \sum_{i\in I} x_i \ge z_I \text{ for } I\subset[n]}$$
for a family $\{z_I\}_{I\subset[n]}\in\R^{2^{[n]}}$ such that $z_I + z_J \le z_{I \cup J} + z_{I \cap J}$ for all $I,J \subset [n]$.
\end{definition}

In other words, a generalized permutahedron is obtained as a deformation of the classical permutahedron by moving its facets while keeping the direction of their normal vectors and staying in its deformation cone~\cite{PostnikovReinerWilliams}.
This family of polytopes contains many relevant families of combinatorial polytopes: permutahedra, associahedra, cyclohedra (and more generally, all graph-asso\-ciahedra~\cite{CarrDevadoss}), etc.

Since $\sGP{Z}{z_I} + \sGP{Z}{z'_I} = \sGP{Z}{z_I+z'_I}$, the class of generalized associahedra is closed by Minkowski sum and difference (a Minkowski difference $P-Q$ of two polytopes $P,Q\subset\R^n$ is defined only if there exists a polytope $R$ such that ${P = Q + R}$).
Consequently, for any $\{y_I\}_{I\subset[n]}\in\R^{2^{[n]}}$, the Minkowski sum and difference
$$\GP{Y}{y_I} \eqdef \sum_{I\subset[n]} y_I\simplex_I$$
of faces $\simplex_I \eqdef \conv\set{e_i}{i\in I}$ of the standard simplex $\simplex_{[n]}$ is a generalized permutahedron. Reciprocally, it turns out that any generalized permutahedron $\sGP{Z}{z_I}$ can be decomposed as such a Minkowski sum and difference $\sGP{Y}{y_I}$, and that $\{y_I\}$ is derived from $\{z_I\}$ by M\"obius inversion when all the inequalities defining $\sGP{Z}{z_I}$ are tight:

\begin{proposition}[\cite{Postnikov, ArdilaBenedettiDoker}]\label{prop:minkowskiSumDiff}
Every generalized permutahedron can be written uniquely as a Minkowski sum and difference of faces of the standard simplex:
$$\GP{Z}{z_I} = \GP{Y}{y_I}, \quad \text{where} \quad y_I = \sum_{J \subset I} (-1)^{|I\ssm J|}z_J$$
if all inequalities $\sum_{i\in I} x_i\ge z_I$ are tight.
\end{proposition}

\begin{example}
The classical permutahedron can be written as
$$\Pi_n = \conv\set{(\sigma(1), \dots, \sigma(n))^T}{\sigma\in\mathfrak{S}_n} = \GP{Z}{\frac{|I|(|I|+1)}{2}} = \sum_{|I| = 2} \simplex_I.$$
\end{example}

The Minkowski decomposition of Proposition~\ref{prop:minkowskiSumDiff} is useful in particular to compute the volume of the generalized permutahedra~\cite{Postnikov}.

All our brick polytopes are generalized permutahedra (as Minkowski sums of matroid polytopes).
It raises three questions about our construction:

\begin{question}
Which generalized permutahedra are brick polytopes?
\end{question}

For example, we obtain all graphical zonotopes (Example~\ref{exm:duplicated1}) and all associahedra (Section~\ref{sec:associahedra}). 
However, the only $2$-dimensional brick polytopes are the square, the pentagon and the hexagon: no brick polytope is a triangle.

\begin{question}\label{qu:coeffsMinkowskiSumDiff}
How to compute efficiently the coefficients $\{y_I\}$ in the Minkowski decomposition of a brick polytope~$\Omega(\cN)$ into dilates of faces of the standard simplex?
\end{question}

Lange studies this question in~\cite{Lange} for all associahedra of Hohlweg and Lange~\cite{HohlwegLange} (see also Remark~\ref{rem:minkowskiSumAssociahedra}). In general, observe that the Minkowski sum decomposition we obtain in Proposition~\ref{prop:minkowskiSum} has only positive coefficients, but its summands are matroid polytopes which are not necessarily simplices. In contrast, the Minkowski decomposition of Proposition~\ref{prop:minkowskiSumDiff} has nice summands but requires sums and differences. In general, the two decompositions of Propositions~\ref{prop:minkowskiSum} and~\ref{prop:minkowskiSumDiff} are therefore different. This last observation raises an additional question:

\begin{question}
Which generalized permutahedra can be written as a Minkowski sum of matroid polytopes?
\end{question}


\section{Hohlweg \& Lange's associahedra, revisited}\label{sec:associahedra}

In this section, we first recall the duality between triangulations of a convex polygon and pseudoline arrangements supported by the $1$-kernel of a reduced alternating sorting network~\cite{PilaudPocchiola}. Based on this duality, we observe that the brick polytopes of these particular networks specialize to Hohlweg \& Lange's realizations of the associahedron~\cite{HohlwegLange}.

\subsection{Duality}\label{subsec:duality}

Remember that we call reduced alternating sorting network any network with $n$ levels and ${n \choose 2}$ commutators and such that the commutators adjacent to each intermediate level are alternatively located above and below it. Such a network supports a unique pseudoline arrangement, whose first and last pseudolines both touch the top and the bottom level, and whose intermediate pseudolines all touch either the top or the bottom level.

To a word~$x\in\{a,b\}^{n-2}$, we associate two dual objects~---~see \fref{fig:alternating}:

\begin{enumerate}
\item $\cN_x$ denotes the reduced alternating sorting network such that the $(i+1)$\textsuperscript{th} pseudoline touches its top level if $x_i=a$ and its bottom level if $x_i=b$, for all~${i\in[n-2]}$.
\item $\cP_x$ denotes the $n$-gon obtained as the convex hull of $\set{p_i}{i\in[n]}$ where $p_1=(1,0)$, $p_n=(n,0)$ and $p_{i+1}$ is the point of the circle of diameter~$[p_1,p_n]$ with abscissa~$i+1$ and located above~$[p_1,p_n]$ if $x_i=a$ and below~$[p_1,p_n]$ if~$x_i=b$, for all $i\in[n-2]$.
\end{enumerate}
For any distinct~$i,j\in[n]$, we naturally both label by~$\{i,j\}$ the diagonal $[p_i,p_j]$ of~$\cP_x$ and the commutator of~$\cN_x$ where cross the $i$\textsuperscript{th} and $j$\textsuperscript{th} pseudolines of the unique pseudoline arrangement supported by~$\cN_x$. Note that the commutators incident to the first and last level of~$\cN_x$ correspond to the edges of the convex hull of~$\cP_x$.

In \fref{fig:alternating}, we have represented~$\cN_x$ and~$\cP_x$ for $x\in\{bbb, aab, aba\}$. The five missing reduced alternating sorting networks with $5$ levels are obtained by reflection of these three with respect to the horizontal or vertical axis.

\begin{figure}[b]
  \capstart
  \centerline{\includegraphics[width=\textwidth]{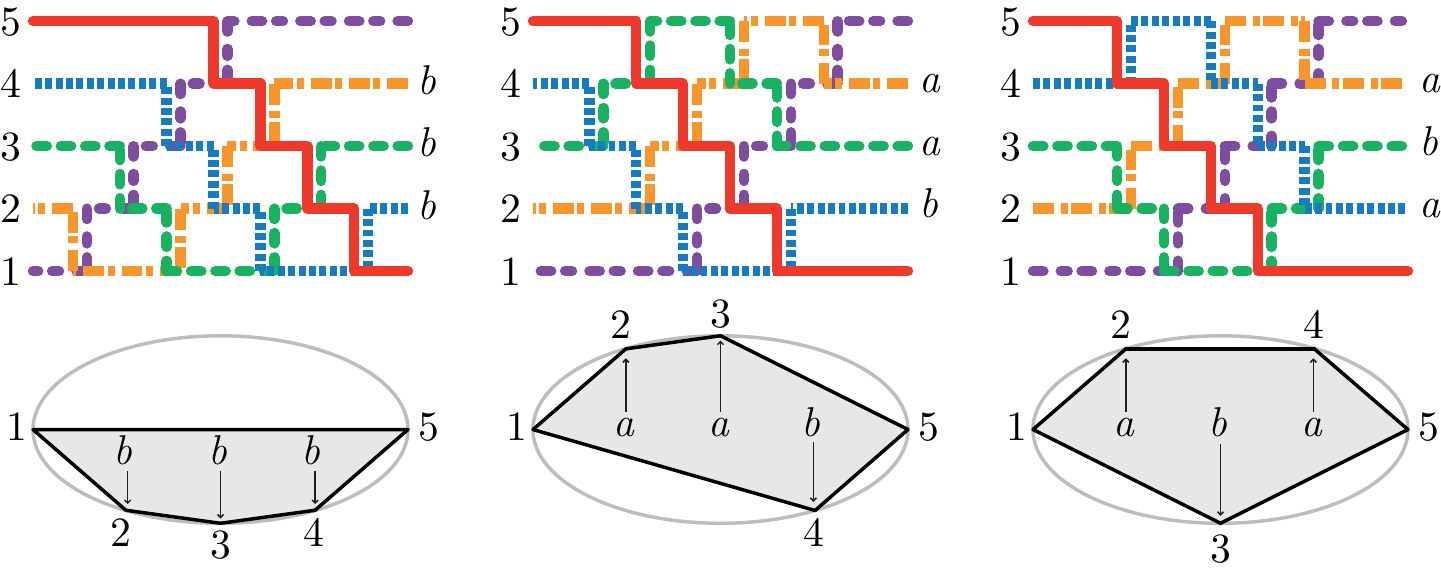}}
  \caption{The networks~$\cN_x$ and the polygons~$\cP_x$ for the words $bbb$, $aab$ and $aba$. The leftmost sorting network is the ``bubble sort'', while the rightmost is the ``even-odd transposition sort''.}
  \label{fig:alternating}
\end{figure}

We call \defn{$1$-kernel} of a network~$\cN$ the network~$\cN^1$ obtained from~$\cN$ by erasing its first and last horizontal lines, as well as all commutators incident to them. For a word $x\in\{a,b\}^{n-2}$, the network~$\cN_x^1$ has~$n-2$ levels and~${n \choose 2}-n$ commutators. Since we erased the commutators between consecutive pseudolines on the top or bottom level of~$\cN_x$, the remaining commutators are labeled by the internal diagonals of~$\cP_x$. The pseudoline arrangements supported by~$\cN_x^1$ are in bijection with the triangulations of~$\cP_x$ through the following duality~---~see \fref{fig:dualite}:

\begin{proposition}[\cite{PilaudPocchiola}]\label{prop:duality}
Fix a word~$x\in\{a,b\}^{n-2}$. The set of commutators of~$\cN_x$ labeled by the internal diagonals of a triangulation~$T$ of $\cP_x$ is the set of contacts of a pseudoline arrangement~$T^*$ supported by $\cN_x^1$. Reciprocally the internal diagonals of $\cP_x$ which label the contacts of a pseudoline arrangement supported by~$\cN_x^1$ form a triangulation of $\cP_x$.
\end{proposition}

\begin{figure}
  \capstart
  \centerline{\includegraphics[scale=1]{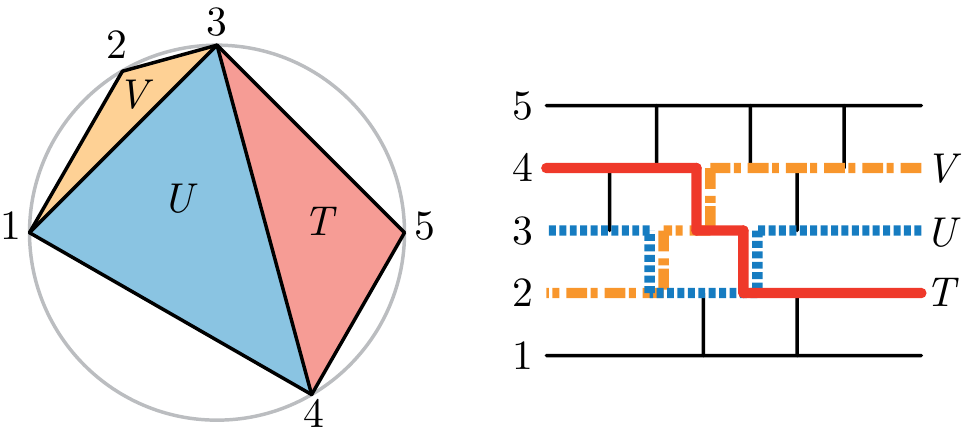}}
  \caption{A triangulation $T$ of $\cP_{aab}$ and its dual pseudoline arrangement $T^*$ on the $1$-kernel of the sorting network~$\cN_{aab}$.}
  \label{fig:dualite}
\end{figure}

The dual pseudoline arrangement $T^*$ of a triangulation $T$ of $\cP_x$ has one pseudoline $\Delta^*$ dual to each triangle $\Delta$ of $T$. A commutator is the crossing between two pseudolines $\Delta^*$ and $\Delta'^*$ of $T^*$ if it is labeled by the common bisector of the triangles $\Delta$ and $\Delta'$ (a \defn{bisector} of $\Delta$ is an edge incident to one vertex of $\Delta$ and which separates its remaining two vertices). A commutator is a contact between $\Delta^*$ and $\Delta'^*$ if it is labeled by a common edge of $\Delta$ and $\Delta'$. Consequently, this duality defines an isomorphism between the graph of flips on pseudoline arrangements supported by~$\cN_x^1$ and the graph of flips on triangulations of~$\cP_x$. Furthermore, we obtain the following interpretation of the contact graph of~$T^*$:

\begin{lemma}\label{lem:tree}
The contact graph $(T^*)^\contact$ of the dual pseudoline arrangement $T^*$ of a triangulation~$T$ is precisely the dual tree of $T$, with some orientations on the edges.
\end{lemma}

\begin{remark}
This duality can be extended to any reduced sorting network. First, a reduced sorting network~$\cN$ with $n$ levels can be seen as the dual arrangement of a set~$\cP$ of $n$ points in a topological plane. Second, the pseudoline arrangements which cover the $1$-kernel of~$\cN$ correspond to the pseudotriangulations of~$\cP$. We refer to~\cite{PilaudPocchiola} or~\cite[Chapter~3]{Pilaud} for details. In Section~\ref{sec:multi}, we recall another similar duality between $k$-triangulations of the $n$-gon and pseudoline arrangements supported by the $k$-kernel of a reduced alternating sorting network with $n$ levels.
\end{remark}

\subsection{Associahedra}

Let $\cN_x$ be a reduced alternating sorting network with $n$ levels. According to Proposition~\ref{prop:duality} and Lemma~\ref{lem:tree}, its $1$-kernel~$\cN_x^1$ is a minimal network: the pseudoline arrangements it supports correspond to triangulations of~$\cP_x$ and their contact graphs are the dual trees of these triangulations (with some additional orientations). Consequently, the brick polytope~$\Omega(\cN_x^1)$ is a simple $(n-3)$-dimensional polytope whose graph is isomorphic to the graph of flips~$G(\cN_x^1)$. Since this graph is isomorphic to the graph of flips on triangulations of the polygon~$\cP_x$, we obtain many realizations of the $(n-3)$-dimensional associahedron with integer coordinates:

\begin{proposition}
For any word~$x\in\{a,b\}^{n-2}$, the simplicial complex of crossing-free sets of internal diagonals of the convex $n$-gon~$\cP_x$ is (isomorphic to) the boundary complex of the polar of~$\Omega(\cN_x^1)$. Thus, the brick polytope~$\Omega(\cN_x^1)$ is a realization of the $(n-3)$-dimensional associahedron.
\end{proposition}

It turns out that these polytopes coincide up to translation with the associahedra of Hohlweg \& Lange~\cite{HohlwegLange}. For completeness, let us  recall their construction (we adapt notations to fit our presentation). Consider the polygon~$\cP_x$ associated to a word~$x\in\{a,b\}^{n-2}$, and let~$T$ be a triangulation of~$\cP_x$. For~$j \in [n-2]$, there is a unique triangle~$\Delta_j(T)$ of~$T$ with vertices~$u < j+1 < v$. Let~$\pi_j(T)$ denote the product of the number of edges of~$\cP_x$ between $u$ and~$j+1$ by the number of edges of~$\cP_x$ between~$j+1$ and~$v$. Associate to the triangulation~$T$ the vector~$\omega(T)$ whose $j$\textsuperscript{th} coordinate is~$\pi_j(T)$ if $x_{j+1} = b$ and~$n+1-\pi_j(T)$ if $x_{j+1} = a$. The associahedra of Hohlweg \& Lange~\cite{HohlwegLange} is the convex hull of the vectors~$\omega(T)$ associated to the $\frac{1}{n-1}{2n-4 \choose n-2}$ triangulations of~$\cP_x$. It is straightforward to check that our duality from~$T$ to~$T^*$ maps~$\Delta_j(T)$ to the $j$\textsuperscript{th} pseudoline of~$T^*$, and the vector~$\omega(T)$ to our brick vector~$\omega(T^*)$, up to a constant translation.

\begin{figure}[b]
  \capstart
  \centerline{\includegraphics[width=1\textwidth]{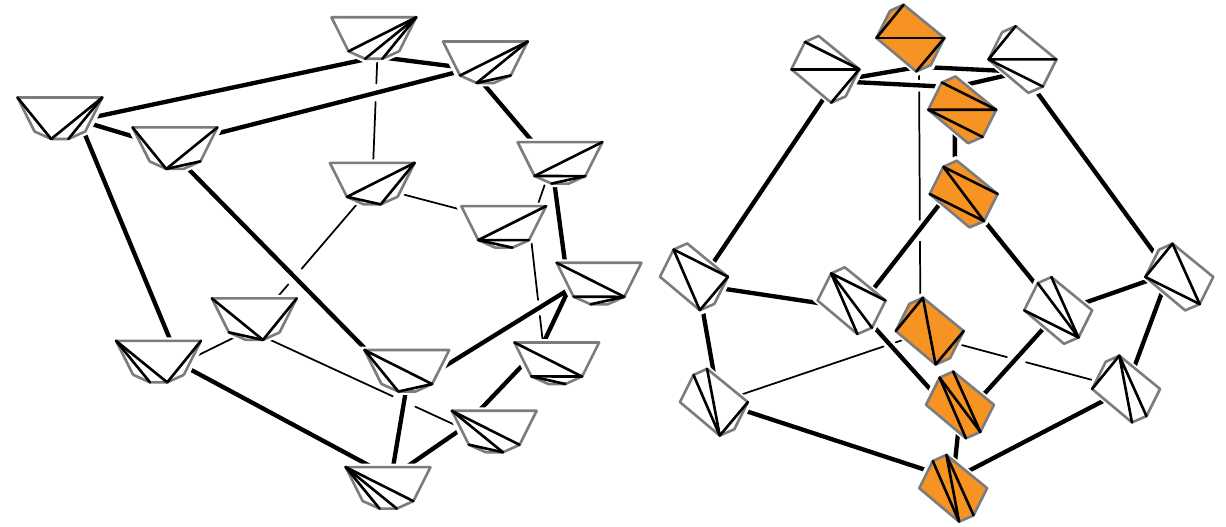}}
  \caption{The brick polytopes~$\Omega(\cN_{b^4}^1)$ (left) and $\Omega(\cN_{a^2b^2}^1)$ (right) provide two different realizations of the $3$-dimensional associahedron. The convex hull of the brick vectors of the centrally symmetric triangulations of~$\cP_{a^2b^2}$ (colored in the picture) is a realization of the $2$-dimensional cyclohedron.}
  \label{fig:associahedra2B}
\end{figure}

Observe that the associahedron~$\Omega(\cN_x^1)$ does not depend on the first and last letters of~$x$ since we erase the first and last levels of~$\cN_x$. Furthermore, a network~$\cN_x$ and its reflection~$v(\cN_x)$ (resp.~$h(\cN_x)$) through the vertical (resp.~horizontal) axis give rise to affinely equivalent associahedra. Affine equivalence between these associahedra is studied in~\cite{bhlt}. Two non-affinely equivalent $3$-dimensional associahedra are presented in \fref{fig:associahedra2B}.

\begin{example}\label{exm:loday}
We obtain Loday's realization of the $(n-3)$-dimensional associahedron~\cite{Loday} as a translate of the brick polytope of the $1$-kernel of the ``bubble sort'' network~$\cB_n \eqdef \cN_{b^{n-2}}$ associated to the word $b^{n-2}$~---~see \fref{fig:associahedra2B} (left).
\end{example}

We now describe normal vectors for the facets of these associahedra. For any word ${x\in\{a,b\}^{n-2}}$, the facets of the brick polytope~$\Omega(\cN_x^1)$ are in bijection with the commutators of~$\cN_x^1$. The vertices of the facet corresponding to a commutator~$\gamma$ are the brick vectors of the pseudoline arrangements supported by~$\cN_x^1$ and with a contact at~$\gamma$. We have already seen that a normal vector of this facet is given by the characteristic vector of the cut induced by~$\gamma$ in the contact graphs of the pseudoline arrangements supported by~$\cN_x^1$ and with a contact at~$\gamma$. In the following lemma, we give an additional description of this characteristic vector:

\begin{lemma}\label{lem:facetsAssociahedra}
Let $\Lambda$ be a pseudoline arrangement supported by~$\cN_x^1$ and let $\gamma$ be a contact of~$\Lambda$. The arc corresponding to~$\gamma$ is a cut of the contact graph~$\Lambda^\contact$ which separates the pseudolines of~$\Lambda$ passing above~$\gamma$ from those passing below~$\gamma$.
\end{lemma}

\begin{proof}
Let $T$ be the triangulation of~$\cP_x$ such that $T^*=\Lambda$. Let $\Delta$ and $\Delta'$ be two triangles of~$T$ whose dual pseudolines~$\Delta^*$ and~$\Delta'^*$ pass respectively above and below~$\gamma$. Then $\Delta$ and $\Delta'$ are located on opposite sides of the edge~$\gamma$ of~$T$, and thus, they cannot share an edge, except if it is~$\gamma$ itself. Consequently, their dual pseudolines~$\Delta^*$ and~$\Delta'^*$ cannot have a contact, except if it is~$\gamma$ itself. We obtain that~$\gamma$ is the only contact between the pseudolines of~$\Lambda$ passing above~$\gamma$ and those passing below~$\gamma$. This implies the lemma.
\end{proof}

\begin{remark}
The pseudolines passing below~$\gamma$ in any pseudoline arrangement supported by~$\cN_x^1$ with a contact at~$\gamma$ are also the pseudolines of the greedy pseudoline arrangement~$\Gamma(\cN_x^1)$ which pass below the cell immediately to the left of~$\gamma$. This provides a fast method to obtain a normal vector for each facet of $\Omega(\cN_x^1)$~---~see \fref{fig:normalvectors} for illustration.

\begin{figure}[h]
  \capstart
  \centerline{\includegraphics[scale=1]{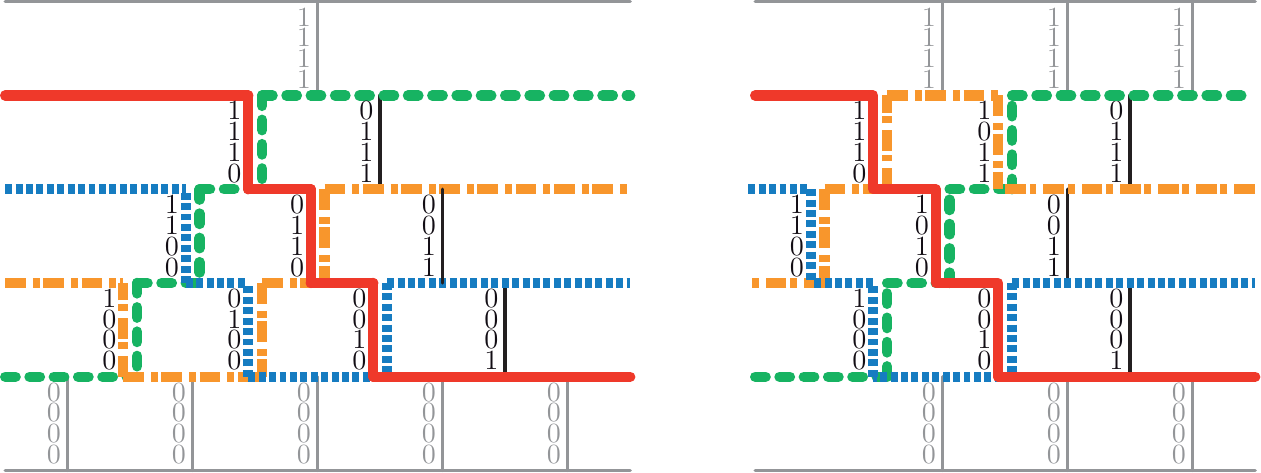}}
  \caption{Normal vectors for the facets of the associahedra~$\Omega(\cN_{b^4}^1)$ and~$\Omega(\cN_{a^2b^2}^1)$, read on the greedy pseudoline arrangements~$\Gamma(\cN_{b^4}^1)$ and~$\Gamma(\cN_{a^2b^2}^1)$.}
  \label{fig:normalvectors}
\end{figure}
\end{remark}

\begin{corollary}
For any~$x\in\{a,b\}^{n-2}$, the brick polytope~$\Omega(\cN_x^1)$ has $n-3$ pairs of parallel facets. The diagonals of~$\cP_x$ corresponding to two parallel facets of~$\Omega(\cN_x^1)$ are crossing.
\end{corollary}

\begin{figure}
  \capstart
  \centerline{\includegraphics[width=1\textwidth]{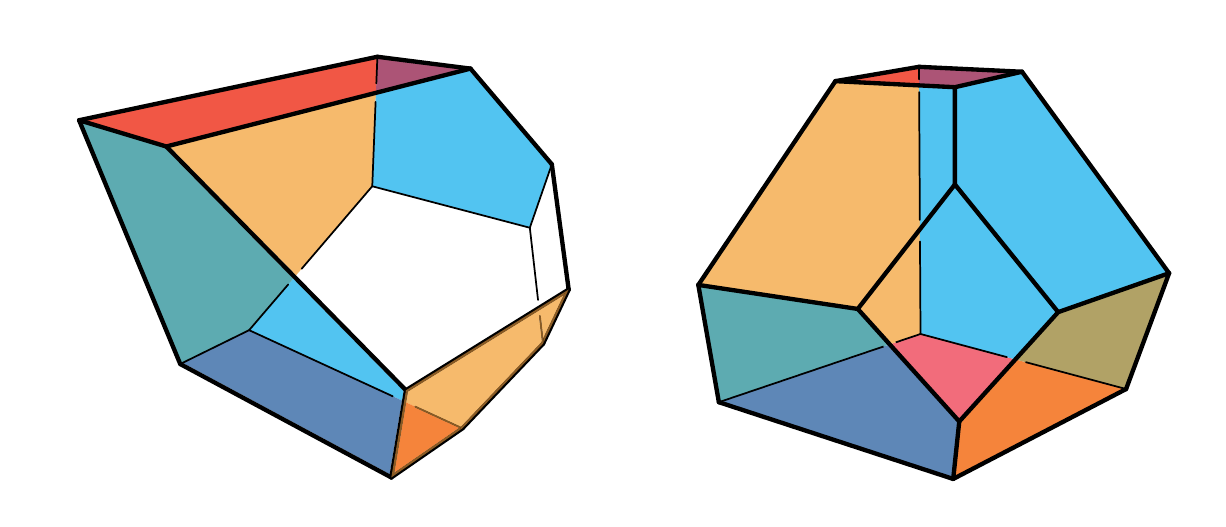}}
  \caption{Pairs of parallel facets in the associahedra of \fref{fig:associahedra2B}. Since they are arranged differently (for example, there is a vertex adjacent to none of these facets in the rightmost realization), these two associahedra are not affinely equivalent.}
  \label{fig:associahedra2C}
\end{figure}

\begin{proof}
For $i\in[n-3]$, the normal vector of the facet corresponding to the leftmost contact between the $i$\textsuperscript{th} and $(i+1)$\textsuperscript{th} level of~$\cN_x^1$ is always the vector $f_i \eqdef \sum_{j=1}^i e_j$ while the normal vector of the facet corresponding to the rightmost contact between the $(n-i-2)$\textsuperscript{th} and $(n-i-1)$\textsuperscript{th} level of~$\cN_x^1$ is always the vector $\sum_{j=i+1}^{n-2} e_j = \one - f_i$ ---~see \fref{fig:normalvectors}. Since~$\Omega(\cN_x^1)$ is orthogonal to~$\one$, these two facets are thus parallel. We obtain $n-3$ pairs of parallel facets when $i$ varies from $1$ to $n-3$. Finally, since two parallel facets of~$\Omega(\cN_x^1)$ have no vertex in common, the corresponding diagonals of~$\cP_x$ are necessarily crossing each other.
\end{proof}

\begin{remark}\label{rem:minkowskiSumAssociahedra}
As discussed in Sections~\ref{subsec:minkowskiSum} and~\ref{subsec:generalizedPermutahedra}, the associahedron~$\Omega(\cN_x^1)$ has two different Minkowski decompositions: as a positive Minkowski sum of the polytopes~$\Omega(\cN_x^1,b)$ associated to each brick~$b$ of~$\cN_x^1$, or as a Minkowski sum and difference of faces of the standard simplex~$\simplex_{[n-2]}$.

In Loday's associahedron (\ie when~$x=b^{n-2}$ and~$\cN_x = \cB_n$), these two decompositions coincide. Indeed, for any $i,j\in[n]$ with $j \ge i+3$, denote by $b(i,j)$ the brick of~$\cB_n$ located immediately below the contact between the $i$\textsuperscript{th} and the $j$\textsuperscript{th} pseudoline of the unique pseudoline arrangement supported by~$\cB_n$. Then the Minkowski summand~$\Omega(\cB_n,b(i,j))$ is the face $\simplex_{\{i,\dots,j-2\}}$ of the standard simplex (up to a translation of vector~$\one_{\{1,\dots,i-1\}\cup\{j-1,\dots,n-2\}}$). This implies that
$$\Omega(\cB_n^1) = \sum_{1\le i<j\le n-2} \simplex_{\{i,\dots,j\}}$$
up to translation, and by unicity, that the coefficient $y_I$ is $1$ if $I$ is an interval of $[n-2]$ and $0$ otherwise (singletons are irrelevant since they only involve translations).

For general~$x$, the Minkowski summands~$\Omega(\cN_x^1,b)$ are not always simplices. In~\cite{Lange}, Lange computes the coefficients $\{y_I\}$ in the Minkowski decomposition of any associahedra~$\Omega(\cN_x^1)$ into dilates of faces of the standard simplex, and therefore answers Question~\ref{qu:coeffsMinkowskiSumDiff} for those special networks.
\end{remark}

\begin{remark}\label{rem:cyclohedra}
To close this section, we want to mention that we can similarly present Hohlweg \& Lange's realizations of the cyclohedra~\cite{HohlwegLange}. Namely, consider an antisymmetric word $x\in\{a,b\}^{2n-2}$ (\ie which satisfies $\{x_i,x_{2n-1-i}\}=\{a,b\}$ for all $i$), such that the $(2n)$-gon~$\cP_x$ is centrally symmetric. Then the convex hull of the brick vectors of the dual pseudoline arrangements of all centrally symmetric triangulations of~$\cP_x$ is a realization of the $(n-1)$-dimensional cyclohedron. For example, the centrally symmetric triangulations of~$\cP_{a^2b^2}$ are colored in the right associahedron of \fref{fig:associahedra2B}: the convex hull of the corresponding vertices is a realization of the $2$-dimensional cyclohedron.
\end{remark}


\section{Brick polytopes and multitriangulations}\label{sec:multi}

This last section is devoted to the initial motivation of this work. We start recalling some background on multitriangulations, in particular the duality between $k$-triangulations of a convex polygon and pseudoline arrangements supported by the $k$-kernel of a reduced alternating sorting network. We then describe normal vectors for the facets of the brick polytope of the $k$-kernel of the ``bubble sort network''. Finally, we discuss the relationship between our brick polytope construction and the realization of a hypothetical ``multiassociahedron''.

\subsection{Background on multitriangulations}\label{subsec:backgroundMulti}

We refer the reader to~\cite{PilaudSantos} and \mbox{\cite[Chapters 1--4]{Pilaud}} for detailled surveys on multitriangulations and only recall here the properties needed for this paper. We start with the definition:

\begin{definition}
A \defn{$k$-triangulation} of a convex polygon is a maximal set of its diagonals such that no $k+1$ of them are mutually crossing.
\end{definition}

Multitriangulations were introduced by Capoyleas \& Pach \cite{CapoyleasPach} in the context of extremal theory for geometric graphs: a $k$-triangulation of a convex polygon~$\cP$ induces a maximal $(k+1)$-clique-free subgraph of the intersection graph of the diagonals of~$\cP$. Fundamental combinatorial properties of triangulations (which arise when~$k=1$) extend to multitriangulations~\cite{Nakamigawa,dkm,Jonsson,PilaudSantos}. We restrict the following list to the properties we need later on:

\paragraph{Diagonals~\cite{Nakamigawa, dkm, PilaudSantos}}
Any $k$-triangulation of a convex $n$-gon~$\cP$ has precisely ${k(2n-2k-1)}$ diagonals. A diagonal of~$\cP$ is said to be \defn{\mbox{$k$-relevant}} (resp. \defn{$k$-boundary}, resp.~\defn{$k$-irrelevant}) if it has at least $k$ vertices of~$\cP$ on each side (resp.~precisely $k-1$ vertices of~$\cP$ on one side, resp.~less than $k-1$ vertices of~$\cP$ on one side). All $k$-boundary and $k$-irrelevant diagonals are contained in all $k$-triangulations of~$\cP$.

\begin{figure}[h]
  \capstart
  \centerline{\includegraphics[scale=1]{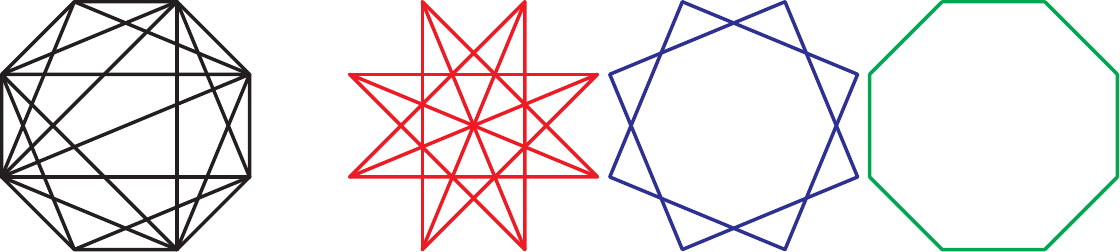}}
  \caption{A $2$-triangulation of the octagon (left) and the \mbox{$2$-rele}\-vant, $2$-boundary and $2$-irrelevant edges of the octogon (right).}
  \label{fig:2triang}
\end{figure}

\paragraph{Stars~\cite{PilaudSantos}}
A \defn{$k$-star} of~$\cP$ is a star-polygon with $2k+1$ vertices $s_0,\dots,s_{2k}$ in convex position joined by the $2k+1$ edges $[s_i,s_{i+k}]$ (the indices have to be understood modulo $2k+1$). Stars in multitriangulations play the same role as triangles in triangulations. In particular, a $k$-triangulation of~$T$ is made up gluing $n-2k$ distinct $k$-stars: a $k$-relevant diagonal of~$T$ is contained in two $k$-stars of~$T$ (one on each side), while a $k$-boundary diagonal is contained in one $k$-star of~$T$~---~see \fref{fig:stars} for an illustration. Furthermore, any pair of $k$-stars of~$T$ has a unique \defn{common bisector}. This common bisector is not a diagonal of~$T$ and any diagonal of~$\cP$ not in~$T$ is the common bisector of a unique pair of $k$-stars.

\begin{figure}[h]
  \capstart
  \centerline{\includegraphics[scale=1]{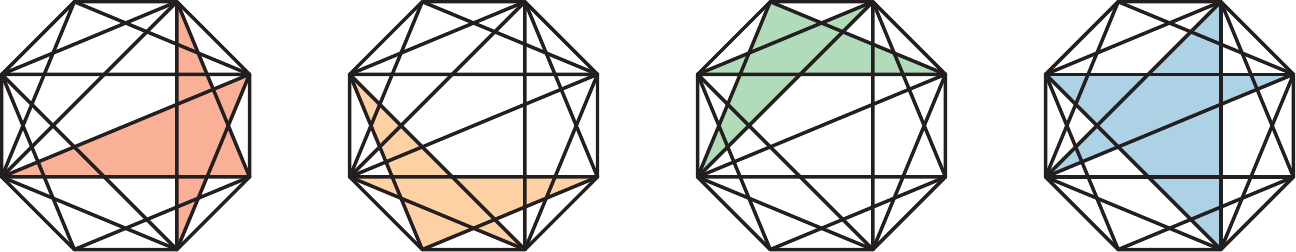}}
  \caption{The four $2$-stars in the $2$-triangulation of \fref{fig:2triang}.}
  \label{fig:stars}
\end{figure}

\paragraph{Flip~\cite{Nakamigawa,Jonsson,PilaudSantos}}
Let $T$ be a $k$-triangulation of~$\cP$, let $e$ be a \mbox{$k$-rele}\-vant edge of~$T$, and let $f$ denote the common bisector of the two $k$-stars of~$T$ containing~$e$. Then $T\symdif\{e,f\}$ is again a $k$-triangulation of~$\cP$, and $T$ and $T\symdif\{e,f\}$ are the only $k$-triangulations of $\cP$ containing $T\ssm\{e\}$. We say that we obtain~$T\symdif\{e,f\}$ by the \defn{flip} of~$e$ in~$T$. The graph of flips is $k(n-2k-1)$-regular and connected.

\begin{figure}[h]
  \capstart
  \centerline{\includegraphics[scale=1]{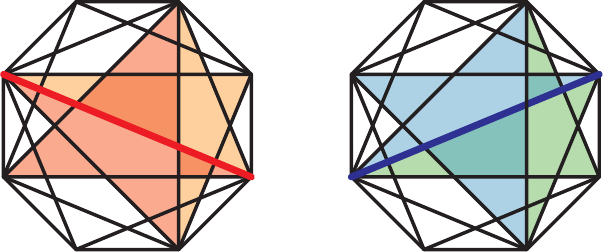}}
  \caption{A flip in the $2$-triangulation of \fref{fig:2triang}.}
  \label{fig:multiFlip}
\end{figure}

\paragraph{Duality~\cite{PilaudPocchiola}}
Fix a word~$x\in\{a,b\}^{n-2}$ and consider the $n$-gon~$\cP_x$ and the sorting network~$\cN_x$ defined in Section~\ref{subsec:duality}. We denote by~$\cN_x^k$ the \defn{$k$-kernel} of $\cN_x$, \ie the network obtained from~$\cN_x$ by erasing its first~$k$ and last~$k$ levels together with the commutators incident to them. The remaining commutators of~$\cN_x^k$ are precisely labeled by the $k$-relevant diagonals of~$\cP_x$, which provides a duality between $k$-triangulations of~$\cP_x$ and pseudoline arrangements supported by~$\cN_x^k$. Namely, the set of commutators of~$\cN_x$ labeled by the internal diagonals of a $k$-triangulation~$T$ of $\cP_x$ is the set of contacts of a pseudoline arrangement~$T^*$ supported by $\cN_x^k$. Reciprocally the $k$-relevant diagonals of $\cP_x$ which label the contacts of a pseudoline arrangement supported by~$\cN_x^k$, together with all $k$-irrelevant and $k$-boundary diagonals of~$\cP_x$, form a $k$-triangulation of $\cP_x$.

\begin{figure}[h]
  \capstart
  \centerline{\includegraphics[width=\textwidth]{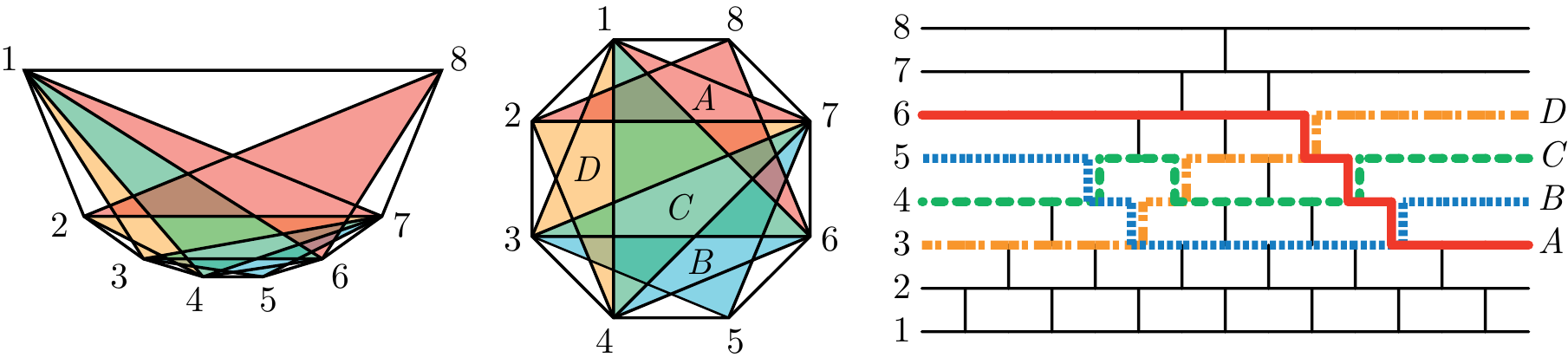}}
  \caption{A $2$-triangulation $T$ of $\cP_{b^6}$ (left), a symmetric representation of it (center), and its dual pseudoline arrangement $T^*$ on the $2$-kernel of the sorting network~$\cN_{b^6}$ (right).}
  \label{fig:dualite2}
\end{figure}

The dual pseudoline arrangement $T^*$ of a $k$-triangulation $T$ of $\cP_x$ has one pseudoline $S^*$ dual to each $k$-star $S$ of $T$. A commutator is the crossing (resp.~a contact) between two pseudolines $S^*$ and $R^*$ of $T^*$ if it is labeled by the common bisector (resp.~by a common edge) of the $k$-stars $S$ and~$R$. Consequently, this duality defines an isomorphism between the graph of flips on pseudoline arrangements supported by~$\cN_x^k$ and the graph of flips on $k$-triangulations of~$\cP_x$.

\subsection{The brick polytope of multitriangulations}

We consider the brick polytope~$\Omega(\cN_x^k)$ of the $k$-kernel of a reduced alternating sorting network~$\cN_x$ (for some word~$x\in\{a,b\}^{n-2}$). Since we erase the first and last $k$ levels, this polytope again does not depend on the first $k$ and last $k$ letters of~$x$. The vertices of this polytope correspond to $k$-triangulations of~$\cP_x$ whose contact graph is acyclic. In contrast to the case of triangulations ($k=1$) discussed in Section~\ref{sec:associahedra},
\begin{enumerate}[(i)]
\item not all $k$-triangulations of~$\cP_x$ appear as vertices of~$\Omega(\cN_x^k)$ and the graph of~$\Omega(\cN_x^k)$ is a proper subgraph of the graph of flips on $k$-triangulations of~$\cP_x$;
\item the combinatorial structure of~$\Omega(\cN_x^k)$ depends upon~$x$.
\end{enumerate}

In the rest of this section, we restrict our attention to the ``bubble sort network''. We denote by~$\cB_n \eqdef \cN_{b^{n-2}}$ and $\cP_n \eqdef \cP_{b^{n-2}}$. For this particular network, we can describe further the combinatorics of the brick polytope~$\Omega(\cB_n^k)$.

\begin{example}
The $f$-vectors of the brick polytopes~$\Omega(\cB_7^2)$, $\Omega(\cB_8^2)$, $\Omega(\cB_9^2)$ and $\Omega(\cB_{10}^2)$ are $(6,6)$, $(22,33,13)$, $(92,185,118,25)$ and $(420,1062,945,346,45)$ respectively. We have represented~$\Omega(\cB_8^2)$ and~$\Omega(\cB_9^2)$ in Figures~\ref{fig:B82} and~\ref{fig:B92}. The polytopes~$\Omega(\cB_7^2)$ and~$\Omega(\cB_8^2)$ are simple while the polytope~$\Omega(\cB_9^2)$ has two non-simple vertices (which are contained in the projection facet of the Schlegel diagram on the right of \fref{fig:B92}) and the polytope~$\Omega(\cB_{10}^2)$ has~$24$ non-simple vertices.

\begin{figure}[p]
	\capstart
	\centerline{\includegraphics[width=1.1\textwidth]{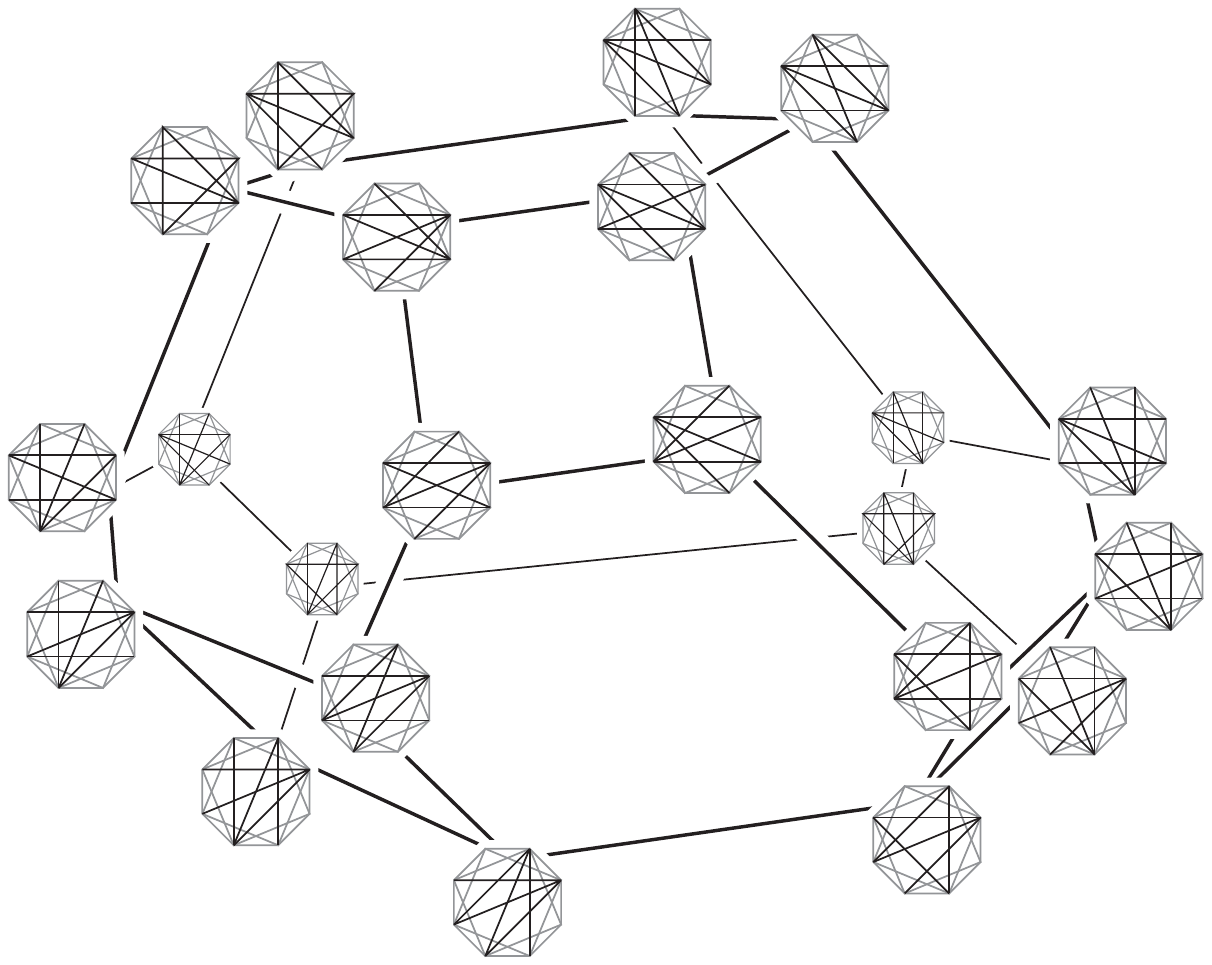}}
	\caption{The $3$-dimensional polytope $\Omega(\cB_8^2)$. Only $22$ of the $84$ $2$-triangulations of the octagon appear as vertices.}
	\label{fig:B82}
\end{figure}

\begin{figure}[p]
	\capstart
	\centerline{\includegraphics[width=\textwidth]{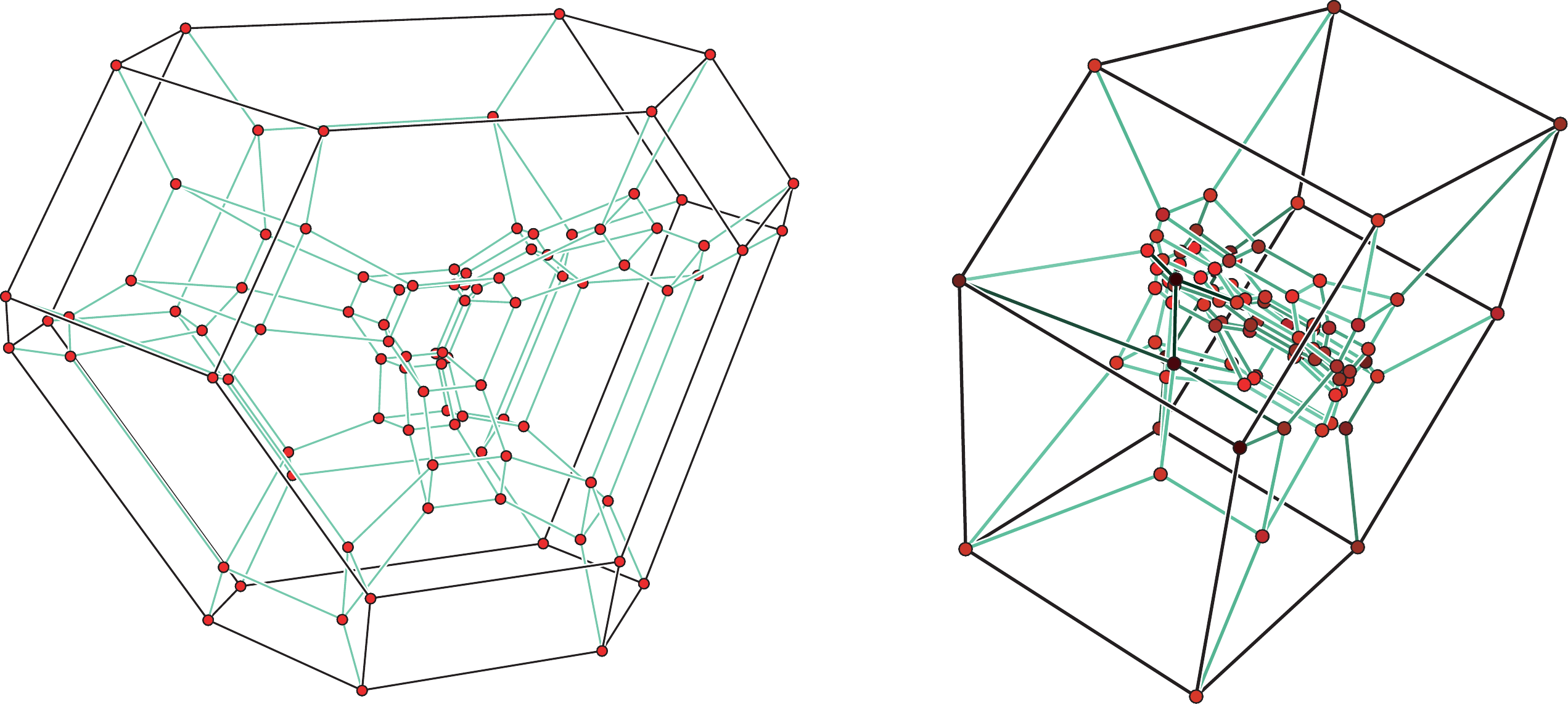}}
	\caption{Two Schlegel diagrams of the $4$-dimensional polytope $\Omega(\cB_9^2)$. On the second one, the two leftmost vertices of the projection facet are non-simple vertices.}
	\label{fig:B92}
\end{figure}
\end{example}

As mentioned in Example~\ref{exm:loday}, the brick polytope~$\Omega(\cB_n^1)$ coincides (up to translation) with Loday's realization of the $(n-3)$-dimensional associahedron~\cite{Loday}~---~see \fref{fig:associahedra2B} (left). The facet normal vectors of this polytope are all vectors of $\{0,1\}^{n-2}$ which are neither $0^{n-2}$ nor $1^{n-2}$ and whose $1$'s are consecutive. More precisely, the vector~$\sum_{\ell=i}^{j-2} e_\ell$ is a normal vector of the facet corresponding to the diagonal $[i,j]$, for any $i,j\in[n]$ with $j \ge i+2$.

For general~$k$, we can similarly provide a more explicit description of the facets of the brick polytope $\Omega(\cB_n^k)$. Representatives for their normal vectors are given by the following sequences:

\begin{definition}
A sequence of~$\{0,1\}^q$ is \defn{$p$-valid} if it is neither $0^q$ nor $1^q$ and if it does not contain a subsequence $10^r1$ for $r\ge p$. In other words, all subsequences of $p$ consecutive zeros appear before the first $1$ or after the last $1$.
\end{definition}

\begin{remark}
Let $\VS_{p,q}$ denote the number of $p$-valid sequences of~$\{0,1\}^q$. It is easy to see that $\VS_{1,q}=\frac{1}{2}(q-1)(q+2)$ (the number of internal diagonals of a $(q+2)$-gon!) and that $\VS_{2,q}=F_{q+4}-(q+4)$, where $F_n$ denotes the $n$\textsuperscript{th} Fibonacci number. To compute $\VS_{p,q}$ in general, consider the non-ambiguous rational expression $0^*(1,10,100,\dots,10^{p-1})^*10^*$. The corresponding rational language consists of all $p$-valid sequences plus all non-empty sequences of~$1$. Thus, the generating function of $p$-valid sequences is:
$$\sum_{q\in\N} \VS_{p,q}x^q=\frac{1}{1-x}\,\frac{1}{1-\sum_{i=1}^{p} x^i}\,x\,\frac{1}{1-x}-\frac{x}{1-x}=\frac{x^2(2-x^p)}{(1-2x+x^{p+1})(1-x)}.$$
\end{remark}

Let $\sigma$ be a sequence of $\{0,1\}^{n-2k}$. Let $|\sigma|_0$ denote the number of zeros in $\sigma$. For all $i\le |\sigma|_0+2k$, we denote by $\zeta_i(\sigma)$ the position of the $i$\textsuperscript{th} zero in the sequence $0^k\sigma0^k$, obtained from $\sigma$ by appending a prefix and a suffix of $k$ consecutive zeros. We associate to $\sigma$ the set of edges of~$\cP_n$:
$$D(\sigma) \eqdef \set{[\zeta_i(\sigma),\zeta_{i+k}(\sigma)]}{i\in[|\sigma|_0+k]}.$$
Observe that $D(\sigma)$ can contain some $k$-boundary edges of~$\cP_n$.

\begin{proposition}\label{prop:multi}
Each $k$-valid sequence $\sigma\in\{0,1\}^{n-2k}$ is a normal vector of a facet $F_\sigma$ of the brick polytope $\Omega(\cB_n^k)$. The facet~$F_\sigma$ contains precisely the brick vectors of the dual pseudoline arrangements of the $k$-triangulations of~$\cP_n$ containing $D(\sigma)$. Furthermore, every facet of $\Omega(\cB_n^k)$ is of the form $F_\sigma$ for some $k$-valid sequence $\sigma\in\{0,1\}^{n-2k}$.
\end{proposition}

\begin{proof}[Proof (sketch)]
Consider a sequence $\sigma\in\{0,1\}^{n-2k}$ and let $(U,V)$ be a partition of $[n-2k]$ such that $\sigma=\one_U=\one-\one_V$. Since $D(\sigma)$ contains no $(k+1)$-crossing, there exists a $k$-triangulation~$T$ of~$\cP_n$ containing $D(\sigma)$. We claim that the $k$-relevant edges of $D(\sigma)$ separate $U$ from $V$ in the contact graph~$(T^*)^\contact$, and that this cut is minimal if and only if $\sigma$ is $k$-valid. This claim together with Corollary~\ref{coro:fd} prove our proposition. We refer the reader to \cite{Pilaud} for details.
\end{proof}

\subsection{Towards a construction of the multiassociahedron?}\label{subsec:multiassociahedron}

Let $\Delta_n^k$ denote the simplicial complex of $(k+1)$-crossing-free sets of $k$-relevant diagonals of a convex $n$-gon. Its maximal elements are $k$-triangulations of the $n$-gon and thus it is pure of dimension $k(n-2k-1)-1$. In an unpublished manuscript~\cite{Jonsson}, Jonsson proved that $\Delta_n^k$ is in fact a shellable sphere (see also the recent preprint~\cite{SerranoStump}). However, the question is still open to know whether the sphere~$\Delta_n^k$ can be realized as the boundary complex of a simplicial $k(n-2k-1)$-dimensional polytope. As a conclusion, we discuss some interactions between this question and our paper.

\paragraph{Universality}
According to Pilaud \& Pocchiola's duality presented in Section~\ref{subsec:backgroundMulti}, this question seems to be only a particular subcase of Question~\ref{qu:polytope}. However, we show in the next proposition that the simplicial complexes $\Delta_n^k$ $(n,k\in\N)$ contain all simplicial complexes $\Delta(\cN)$ ($\cN$ sorting network). If $X$ is a subset of the ground set of a simplicial complex~$\Delta$, remember that the \defn{link} of $X$ in $\Delta$ is the subcomplex~$\set{Y \ssm X}{X \subset Y \in \Delta}$, while the \defn{star} of~$X$ is the join of $X$ with its~link.

\begin{proposition}\label{prop:universality}
For any sorting network~$\cN$ with $n$ levels and $m$ commutators, the simplicial complex~$\Delta(\cN)$ is (isomorphic to) a link of~$\Delta_{n+2m-2}^{m-1}$.
\end{proposition}

\begin{proof}
Label the commutators of~$\cN$ from left to right (and from top to bottom if some commutators lie on the same vertical line). For any~$i\in[m]$, let $i_\square$ be such that the $i$\textsuperscript{th} commutator of~$\cN$ join the $i_\square$\textsuperscript{th} and~$(i_\square+1)$\textsuperscript{th} levels of~$\cN$. Define the set~$W \eqdef \set{\{i,i+i_\square+m-1\}}{i\in[m]}$.

We claim that the simplicial complex~$\Delta(\cN)$ is isomorphic to the simplicial complex of all $m$-crossing-free sets of $(m-1)$-relevant diagonals of the \mbox{$(n+2m-2)$-gon} which contain the diagonals not labeled by~$W$. Indeed, by duality, the \mbox{$(m-1)$-trian}\-gulations of this complex correspond to the pseudoline arrangements supported by~$\cB_{n+2m-2}^{m-1}$ whose sets of contacts contain the contacts of $\cB_{n+2m-2}^{m-1}$ not labeled by~$W$. Thus, our claim follows from the observation that, by construction, the contacts labeled by~$W$ are positioned exactly as those of~$\cN$~---~see \fref{fig:universality}.
\end{proof}

\begin{figure}[h]
	\capstart
	\centerline{\includegraphics[scale=1.2]{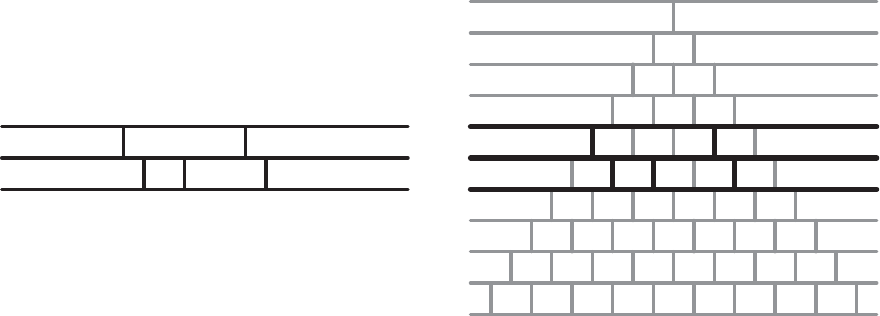}}
	\caption{Universality of the multitriangulations: any sorting network~$\cN$ is a subnetwork of a bubble sort network~$\cB_n^k$ for some parameters~$n$ and~$k$ depending on~$\cN$.}
	\label{fig:universality}
\end{figure}

Since links in shellable spheres are shellable spheres, this proposition extends Jonsson's result~\cite{Jonsson} to any sorting network:

\begin{corollary}\label{coro:shellable}
For any sorting network~$\cN$, the simplicial complex~$\Delta(\cN)$ is a shellable sphere.
\end{corollary}

This corollary is also a consequence of Knutson \& Miller's results on the shellability of any subword complex~\cite{KnutsonMiller}.
Proposition~\ref{prop:universality} would also extend any proof of the polytopality of $\Delta_n^k$ to that of $\Delta(\cN)$:

\begin{corollary}
If the simplicial complex $\Delta_n^k$ is polytopal for any $n,k\in\N$, then the simplicial complex $\Delta(\cN)$ is polytopal for any sorting network~$\cN$.
\end{corollary}

In particular, if one manages to construct a multiassociahedron realizing the simplicial complex~$\Delta_n^k$, it would automatically provide an alternative construction to the polytope of pseudotriangulations~\cite{RoteSantosStreinu-polytope}. We want to underline again that the brick polytope does not realizes the simplicial complex of pointed crossing-free sets of internal edges in a planar point set.

\paragraph{Projections and brick polytopes}
In this paper, we have associated to each \mbox{$k$-trian}gulation~$T$ of the $n$-gon the brick vector~$\omega(T^*) \in \R^{n-2k}$ of its dual pseudoline arrangement~$T^*$ on~$\cB_n^k$. We have seen that the convex hull~$\Omega(\cB_n^k)$ of the set $\set{\omega(T^*)}{T \; k\text{-triangulation of the } n\text{-gon}}$ satisfies the following two properties:
\begin{enumerate}[(i)]
\item The graph of~$\Omega(\cB_n^k)$ is (isomorphic to) a subgraph of the graph of flips on $k$-triangulations of an $n$-gon (Corollary~\ref{coro:graph}).
\item The set of $k$-triangulations whose brick vector belongs to a given face of~$\Omega(\cB_n^k)$ forms a star of~$\Delta_n^k$ (Proposition~\ref{prop:multi}).
\end{enumerate}
One could thus reasonably believe that our point set could be a projection of the polar of a hypothetical realization of~$\Delta_n^k$. The following observation kills this hope:

\begin{proposition}
Let $P$ be the polar of any realization of~$\Delta_7^2$. It is impossible to project~$P$ down to the plane such that the vertex of~$P$ labeled by each \mbox{$2$-triangulation}~$T$ of the heptagon is sent to the corresponding brick vector~$\omega(T^*)$.
\end{proposition}

\begin{proof}
The network~$\cB_7^2$ is a $3$-level alternating network and thus has been addressed in Example~\ref{exm:3levels}. See \fref{fig:B72} (left) to visualize the brick polytope~$\Omega(\cB_7^2)$. The contacts are labeled from left to right, and the brick vector of a pseudoline arrangement covering $\cB_7^2$ is labeled by its four contacts.

\begin{figure}[h]
	\capstart
	\centerline{\includegraphics[scale=.9]{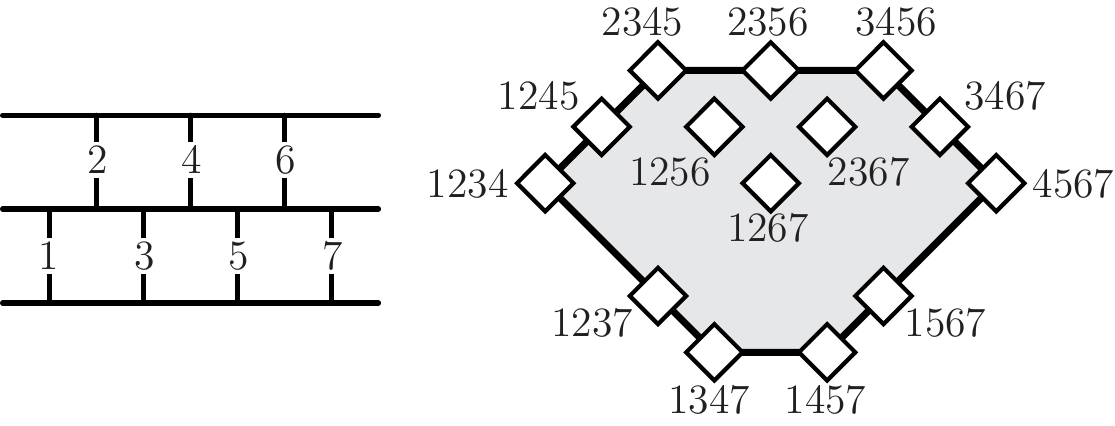}}
	\caption{The $2$-dimensional brick polytope $\Omega(\cB_7^2)$.}
	\label{fig:B72}
\end{figure}

Assume that there exists a polytope~$P$ whose polar realizes~$\Delta_7^2$ and which projects to the point configuration $\set{\omega(T^*)}{T \; 2\text{-triangulation of the heptagon}}$. Observe that the projections of two non-parallel edges of a $2$-dimensional face of~$P$ either are not parallel or lie on a common line (when the $2$-dimensional face is projected to a segment of this line). We will reach a contradiction by considering the facet of~$P$ labeled by $4$. We have represented in \fref{fig:B72bis} the projections of its \mbox{$2$-dimen}\-sional faces: there are two triangles $24$ and $46$ which project to a segment, two quadrilaterals $14$ and $47$ and two pentagons $34$ and $45$. Since they belong to the $2$-dimensional face~$45$ of~$P$, and since they project on two distinct parallel lines, the two edges $(3456,4567)$ and $(1245,1457)$ of~$P$ are parallel. Similarly, the edges $(1245,1457)$ and $(1234,1347)$ are parallel and the edges~$(1234,1347)$ and $(3456,3467)$ are parallel. By transitivity, we obtain that the edges $(3456,4567)$ and $(3456,3467)$ of~$P$ are parallel which is impossible since they belong to the triangular face~$46$.
\begin{figure}
	\capstart
	\centerline{\includegraphics[scale=.9]{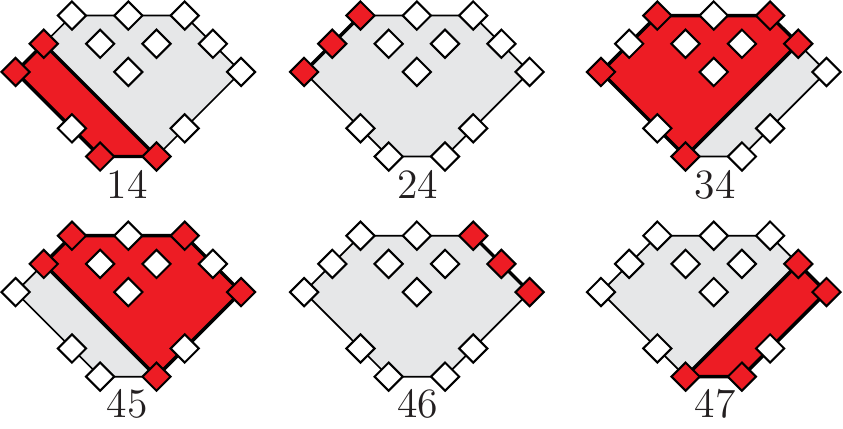}}
	\caption{Projections of some $2$-dimensional faces of~$\Delta_7^2$ on $\Omega(\cB_7^2)$.}
	\label{fig:B72bis}
\end{figure}
\end{proof}


\section*{Acknowledgment}

We thank Christian Stump for helpful comments on a preliminary version of this paper. We are also grateful to Christophe Hohlweg and Carsten Lange for interesting discussion about this paper and related topics. Finally, we thank two anonymous referees for relevant suggestions.


\bibliographystyle{alpha}
\bibliography{biblio.bib}

\end{document}